\documentclass[11pt,amsfonts]{article}
\usepackage[utf8]{inputenc}
\usepackage[T1]{fontenc} 
\usepackage{tikz}
\usepackage[english]{babel} 
\usepackage{mathrsfs}
\usepackage{amssymb}
\usepackage[cmex10]{amsmath}
\usepackage{amsthm}
\usepackage{wrapfig}
\usepackage{latexsym}
\usepackage{amsfonts}
\usepackage{color}
\usepackage[font=small, labelfont=bf]{caption}
\usepackage{ctable}
\usepackage{floatrow}
\usepackage[colorlinks=true, citecolor=red, linkcolor=blue]{hyperref}
\usepackage{graphicx}
\usepackage{subfig}
\usepackage{verbatim}
\usepackage{epstopdf}
\usepackage{float}
\usepackage[symbol]{footmisc}
\usepackage{lipsum}

\usepackage[expansion=false]{microtype}

\usepackage[top=2cm , bottom=2cm , left=2.5cm ,right=2.5cm ]{geometry}

\usepackage{dsfont}

\usepackage{stmaryrd}

\usepackage{cases}

\theoremstyle{plain}
\newtheorem{definition}{Definition}
\newtheorem{theorem}{Theorem}
\newtheorem{corollary}{Corollary}
\newtheorem{proposition}{Proposition}

\newtheorem{lemma}{Lemma}

\newtheorem{remark}{Remark}

\numberwithin{equation}{section}
\numberwithin{theorem}{section}
\numberwithin{proposition}{section}
\numberwithin{definition}{section}
\numberwithin{remark}{section}
\numberwithin{corollary}{section}
\numberwithin{lemma}{section}

\usepackage{graphicx}
\usepackage{footmisc}

\makeatletter
\newcommand*\bigcdot{\mathpalette\bigcdot@{.5}}
\newcommand*\bigcdot@[2]{\mathbin{\vcenter{\hbox{\scalebox{#2}{$\m@th#1\bullet$}}}}}
\makeatother

\begin{document}
	\sloppy
	
	\renewcommand{\thefootnote}{\arabic{footnote}}

	\begin{center}
		{\Large \textbf{Multidimensional McKean–Vlasov SDEs with mean reflection: well-posedness and existence of optimal control }} \\[0pt]
		~\\[0pt] 
		
		Imane Jarni, Ayoub Laayoun\ and Badr Missaoui

	\end{center}
	
	\renewcommand{\thefootnote}{\arabic{footnote}}

	\begin{abstract}
		In this work, we investigate the multidimensional Skorokhod problem for càdlàg processes, where the reflection is subject to a minimality condition depending on the law of the solution. We then apply these results to establish existence and uniqueness for multidimensional McKean–Vlasov stochastic differential equations with mean reflection. Finally, we address the existence of optimal relaxed controls for such equations.
	\end{abstract}

	\section{Introduction}
	Let $\left(\Omega, \mathbb{F},\left(\mathcal{F}_t\right)_{t \leq T},\mathbb{P}\right)$ be a filtered probability space satisfying the usual conditions. Let $(B_t)_{0 \leq t \leq T}$ be an $\mathbb{R}^m$-valued Brownian motion, and let $(\tilde{N}_t)_{0 \leq t \leq T}$ be a compensated Poisson random measure on $\mathbb{R}^d \backslash\{0\}$ associated with a Lévy measure $\lambda$ independent of $(B_t)_{0 \leq t \leq T}$, and corresponding to a standard Poisson random measure $N(t,A)$. For any measurable subset $A$ in $\mathbb{R}^d \backslash\{0\}$ such that $\lambda(A)<\infty$, we have $$\tilde{N}_t(A)=N(t,A)-t \lambda(A),$$ where $N(t, A)$ satisfies $\mathbb{E}[N(t, A)]=t \lambda(A)$. The Lévy measure $\lambda(\cdot)$ is assumed to be a $\sigma$-finite measure on $\mathbb{R}^d \backslash\{0\}$ and to satisfy the integrability condition: $$\int_{\mathbb{R}^d \backslash\{0\}}\left(1 \wedge|z|^2\right) \lambda(\mathrm{d} z)<\infty.$$ 
	Let $\mathcal{D}$ be a convex subset of $\mathbb{R}^d$. For $t\in[0, T]$, we are interested in studying the following system of stochastic differential equations (SDEs) with constraint on the law of the solutions:
	$$(\Gamma)\left\{\begin{aligned}
		&X_t= X_{0}+ \int_{0}^{t} b\left(s, X_s, \mathcal{L}_{X_s}\right) \mathrm{d} s+ \int_{0}^{t} \sigma\left(s, X_s, \mathcal{L}_{X_s}\right) \mathrm{d} B_s \\
		& \quad\quad	+\int_{0}^{t}\int_{\mathbb{R}^d\backslash\{0\}}\beta\left(s,  X_{s^-}, \mathcal{L}_{X_{s^-}},z\right)  \tilde{N}(\mathrm{d} s, \mathrm{d} z)+ k_t, \quad t \in[0, T], \\
		&
		\mathbb{E}(\phi_t(X_t))\in \mathcal{D}, \\
		& \text{for every} \;z\in \mathcal{D},\; \int_0^t \langle \mathbb{E}(\phi_s(X_s)) - z,  \mathrm{d}k_s \rangle \leq 0
	\end{aligned}
	\right.
	$$
	where $\phi: [0,T]\times\mathbb{R}^d\to\mathbb{R}^d$.
	
	This type of reflection, where the constraint is imposed on the law of the process, was introduced  by Briand, Elie and Hu in the one-dimensional case for Backward stochastic differential equations (BSDEs) in \cite{briand18}. Later, Briand, Ghannoum, and Labart extended this type of reflection to SDEs with jumps in one-dimensional case in \cite{briand20}. Falkowski and Słomiński studied similar mean reflection problems with two constraints in the context of both SDEs and BSDEs in \cite{falkowski21} and \cite{falkowski22}. In the multi-dimensional setting, Briand et al. investigated SDEs with mean reflection in \cite{briand2}, where the function $\phi$ maps $\mathbb{R}^d$ to $\mathbb{R}$, meaning that the constraint domain is an interval in $\mathbb{R}$; their equations are driven by Brownian motion. Qu and Wang studied multi-dimensional BSDEs driven by Brownian motion with mean reflection in a multidimensional domain, in the case where $\phi$ is the identity map, see \cite{qu23}. Qian extended this to the case of jumps and convex domains with smooth boundaries in \cite{qian25}. In this work, we extend the aforementioned studies \cite{briand2}, \cite{qu23} and \cite{qian25} in the sense that we study McKean–Vlasov SDEs of Wiener–Poisson type, where $\phi$ is $\mathbb{R}^d$- valued mapping (not necessary the identity), and where the reflection domain is convex but not necessarily smooth. 
	
	In the first part of this paper, we study the Skorokhod problem with mean minimality condition for inputs that are right-continuous with left limits (càdlàg processes). We establish estimates for the solutions, which we then use to solve the system ($\Gamma$). Another result we present is a particle approximation of ($\Gamma$) by the following:

$$
\left\{
\begin{aligned}
	&X_t^{i,N}= X_{0}^i
	+ \int_0^t b\left(s, X_s^{i,N}, \tfrac{1}{N}\sum_{j=1}^N \delta_{X_s^{j,N}}\right)\,\mathrm{d}s
	+ \int_0^t \sigma\left(s, X_s^{i,N}, \tfrac{1}{N}\sum_{j=1}^N \delta_{X_s^{j,N}}\right)\,\mathrm{d}B_s^i \\
	&\quad\quad
	+ \int_0^t\int_{\mathbb{R}^d\setminus\{0\}}
	\beta\left(s, X_{s^-}^{i,N}, \tfrac{1}{N}\sum_{j=1}^N \delta_{X_{s^-}^{j,N}},z\right)
	\tilde{N}^i(\mathrm{d}s, \mathrm{d}z)
	+ K_t^{N}, \\
	&\tfrac{1}{N}\sum_{i=1}^N\phi_t(X_t^{i,N})\in \mathcal{D}, \\
	&\text{for every } z\in\mathcal{D},\;
	\int_0^t \langle 
	\tfrac{1}{N}\sum_{i=1}^N\phi_s(X_s^{i,N}) - z,\,
	\mathrm{d}K_s^{N}
	\rangle \leq 0
\end{aligned}
\right.
$$

	for $N\geq 1$ and $i\in\{1,\dots,N\}$, where $(X_{0}^i,B^i,\tilde{N}^i)$ are independent copies of $(X_0,B,\tilde{N})$. In classical particle approximations of McKean-Vlasov-type dynamics with reflection, the constraint in the particle system usually takes a similar form as in the limiting equation (see e.g., \cite{Adam},\cite{jarni25},\cite{sznitman}), which makes the computations somewhat simpler. The situation is different in the case of mean reflection. Indeed, in this setting, the constraint involves an expectation, which must also be approximated. This is natural, but it introduces additional computational complexity (see e.g., \cite{briand2}, \cite{briand20}). 
	In our work, the key idea behind both the existence and uniqueness of the Skorokhod problem as well as the particle approximation result, lies in the use of a bijection between the mean reflection problem and a Skorokhod problem in a time-dependent domain introduced in \cite{jarni25}. 
	
	In the second part of this paper, we address a stochastic optimal control problem. We consider a controlled McKean–Vlasov SDE with mean reflection
	\begin{equation*}		\begin{cases}
			X_t= x+ \displaystyle\int_{0}^{t} b(s, X_s,\mathcal{L}_{X_s}, u_s)\mathrm{d}s + \displaystyle\int_{0}^{t}\sigma(s,X_s, \mathcal{L}_{X_s}) \mathrm{d}B_s  \\
			\quad\quad+\displaystyle\int_{0}^{t}\int_{\mathbb{R}^d\backslash\{0\}}\beta\left(s,  X_{s^-}, \mathcal{L}_{X_{s^-}},z\right)  \tilde{N}(\mathrm{d} s, \mathrm{d} z)+ k_t,\\
			\mathbb{E}(X_t) \in \mathcal{D}, \\
			\text{for all}\;z\in \mathcal{D}, \displaystyle\int_0^t \langle \mathbb{E}(X_{s}) - z,  \mathrm{d}k_s \rangle \leq 0 . 
		\end{cases}
	\end{equation*}
	Here, the control process $(u_t)_{t\in[0,T]}$ takes values in a given action space $\mathbb{A}$. The objective is to minimize the cost functional
	\begin{equation*}
		J(u) = \mathbb{E}\left(\int_0^T  f(s, X_s, \mathcal{L}_{X_s}, u_s)\mathrm{d}s + \int_0^T h(s,X_s,\mathcal{L}_{X_s})\mathrm{d}|k|_s + g(X_T)\right)
	\end{equation*}
	over the set of admissible controls, i.e. progressively measurable processes with values in $\mathbb{A}$.
	
	In both deterministic and stochastic control settings, it is well known that an optimal control does not necessarily exist within the admissible space of strict controls, unless certain convexity assumptions are imposed. To overcome this limitation, a larger class $\mathcal{R}$ of admissible controls is introduced, known as \textit{relaxed controls}. In this framework, instead of selecting a specific control value $u_t \in \mathbb{A}$ at each time $t$, the controller chooses a probability measure $ q_t(\mathrm{d}a)$ on $\mathbb{A}$. This approach has been extensively developed in the context of stochastic control for diffusion processes, see e.g., \cite{El karoui}, \cite{haussman11}. An alternative approach via occupation measures was proposed in \cite{kurtz}, in the framework of general controlled martingale problems. 
	
	In the setting of SDEs, the first existence result for an optimal relaxed control, under uncontrolled diffusion, was established by Fleming \cite{Fleming} using compactification techniques. The case of SDEs with a controlled diffusion coefficient was later addressed by El-Karoui et al. \cite{El karoui}, who proved that under appropriate conditions, an optimal relaxed control exists. The existence of a strict optimal control can be guaranteed when the image of the action space $\mathbb{A}$ under the mapping $a\mapsto(b(t,x,a), f(t,x,a))$ is convex. This is the so-called Roxin-type convexity condition. For further details, we refer to \cite{El karoui},\cite{haussman11}, and \cite{Kushner}. In this work, we adopt the relaxed control framework to solve the optimal control problem associated with the mean-reflected McKean–Vlasov SDE, establish the existence of an optimal relaxed control, and show that it is a strict control under the Roxin convexity condition. 
	
	The paper is organized as follows. In section \ref{sk} we study the Skorokhod problem with mean minimality condition, establish existence and uniqueness, and derive estimates that will later be used for the McKean–Vlasov equation ($\Gamma$). Section \ref{sdes} is devoted to the study of existence and uniqueness for the system ($\Gamma$), as well as to the analysis of the propagation of chaos property for ($\Gamma$). In section \ref{control}, we consider a stochastic optimal control problem for a mean-reflected McKean–Vlasov SDE. We prove the existence of an optimal relaxed control and, under the Roxin convexity condition, we show that the optimal relaxed control is in fact strict. We also provide an approximation result for relaxed controls.

	We introduce the following notation.	
	\begin{itemize}
		\item $\mathcal{C}:$ The space of all bounded closed convex subsets of $\mathbb{R}^d$ with nonempty interiors, endowed with the Hausdorff metric $d_H$. That is, for any $D, D^{\prime} \in \mathcal{C}$,
		$$
		d_H\left(D, D^{\prime}\right)=\max \left(\sup _{x \in D} \operatorname{d}\left(x, D^{\prime}\right), \sup _{x \in D^{\prime}} \operatorname{d}(x, D)\right),
		$$
		where $\operatorname{d}(x, D)=\inf _{y \in D}|x-y|$, with $|\cdot|$ denoting the Euclidean norm on $\mathbb{R}^d.$
		\item[$\bullet$] If $E$ is a metric space equipped with a metric $\rho$. We denote by $\mathbb{D}([0,T],E)$ the space of all mappings $y:[0,T]\to E$ with càdlàg trajectories such that $y_0=y_{0^-}$ and $y_T=y_{T^-}$. 
		
		Let $\Lambda$ denote the set of strictly increasing, continuous mappings $\lambda : [0,T]\to[0,T]$ such that $\lambda_0=0$ and $\lambda_T=T$. For $\lambda\in\Lambda$, define $$\|\lambda\|_{\circ} := \sup_{0 \leq s < t \leq T} \left| \log\frac{ \lambda_t - \lambda_s}{t - s} \right|$$
		and $$d_{\circ}(x, y):= \inf_{\lambda \in \Lambda} \left\{ \|\lambda\|_{\circ} \vee \sup_{0 \leq t \leq T}  \rho(x(t), y(\lambda(t))) \right\}.$$
		\item For an $\mathbb{R}^d$-valued function $k$ with finite variation, define $|k|_t=\sum_{i=1}^d|k^i|_t$, where $|k^i|_t$ is the total variation of the component $k^i$ on $[0,t]$.
		\item For an $\mathbb{R}^d$-valued martingale $M=(M^1,\dots,M^d)$, let $[M]$ denotes $\sum_{i=1}^d[M^i]$, where $[M^i]$ represents the quadratic variation of the component $M^i$. 
		\item[$\bullet$] Let $D\in\mathcal{C}$, we denote by $\partial D$ the boundary of $D$. For $x\in\partial D$, $\mathcal{N}_{x}$ denotes the set of inward normal unit vectors at $x$, that is, $\eta \in \mathcal{N}_x$ if and only if for any $y\in D$, $\langle x-y, \eta\rangle \leq 0$. Here, $ \langle\cdot,\cdot \rangle$ denotes the usual inner product in $\mathbb{R}^d$.
		\item $\mathcal{U}$: The space of càdlàg, adapted processes $Y$ such that the family $\{Y_t,\;t\in [0,T]\}$ is uniformly integrable.
		\item For any continuous linear map $u$, we denote by $u^*$ its adjoint. Moreover, the operator norm of $u$ is defined by $\|u\|:=\sup_{x\neq 0}\frac{|u(x)|}{|x|}$.
		\item For a topological space $E$, we denote by $\mathcal{B}(E)$ its Borel $\sigma$-algebra, and by $\mathcal{P}(E)$, the set of probability measures on $(E,\mathcal{B}(E))$. For $p\geq 2$, we define $\mathcal{P}_p(E)$ as the set of all probability measures on $(E,\mathcal{B}(E))$ that have finite $p$-th order moments. 
		For $\mu,\nu\in\mathcal{P}_p(E)$, the Wasserstein distance of order $p$ is defined by the formula:
		$$
		W_{\{d,p\}}(\mu, \nu)=\inf \left\{\left[\int d(x,y)^p \pi(d x, d y)\right]^{1 / p} ; \pi \in \mathcal{P}(E \times E) \text { with marginals } \mu \text { and } \nu\right\}.$$
		Note that the distance considered in $\mathbb{R}^d$ is the Euclidean norm, i.e., $d(x,y)=|x-y|$ for $(x,y)\in\mathbb{R}^{2d}$.
		
		Let $t\in [0,T]$. On the space $\mathbb{D}([0,t],\mathbb{R}^{d})$, we define the distance $d(\omega_1,\omega_2):=d_t(\omega_1,\omega_2)=\sup_{s\in [0,t]}|\omega_1(s)-\omega_2(s)|$ for every $\omega_1,\omega_2\in\mathbb{D}([0,t],\mathbb{R}^{d})$. 
		
		In the following, $W_{\{t,p\}}$ denotes $W_{\{d_{t},p\}}$ on the space $\mathcal{P}_p(\mathbb{D}([0,t],\mathbb{R}^{d}))$, and $W_p$ denotes $W_{\{|\cdot|,p\}}$ on the space $\mathcal{P}_p(\mathbb{R}^{d})$.
		\item For $X\in \mathbb{D}([0,T],\mathbb{R}^{d})$, $\mathcal{L}_{X}$ denotes the law of the process $X$ on $\mathcal{P}(\mathbb{D}([0,T],\mathbb{R}^{d}))$. Similarly,  $\mathcal{L}_{X_t}$ denotes the law of $X_t$ on $\mathcal{P}(\mathbb{R}^d)$, where $t\in [0,T]$.
		\item If $\mu \in\mathcal{P}(\mathbb{D}([0,T],\mathbb{R}^d))$, we denote by $\mu_t$ the push-forward image of the measure $\mu$ under $\pi_t$, and by $\mu_{t^-}$ the push-forward image of the measure $\mu$ under $\pi_{t^-}$. Here, 
		$$\begin{array}{rcl}
			(\pi_t, \pi_{t^-}) : & \mathbb{D}([0,T], \mathbb{R}^d) & \to \mathbb{R}^d \times \mathbb{R}^d, \\ 
			& \omega & \mapsto (\omega_t, \omega_{t^-}).
		\end{array}
		$$
		\item For $a\in\mathbb{R}^d, \delta_{a}$ denotes the Dirac measure at $a$.
		\item We denote the predictable $\sigma$-algebra as $\mathscr{P}$.
		\item For $k,p\geq 1$, we denote as $\mathbb{H}^{p,k}$ the set of $\mathbb{R}^d$-valued, càdlàg, and adapted processes $(X_{t})_{t\geq 0}$ such that $\mathbb{E}\left(\int_{0}^T|X_s|^p\mathrm{d}s\right)<\infty.$
		\item For two stopping times $\tau$ and $\tau^{\prime}$ such that $\tau\leq \tau^{\prime}$, we denote as $\mathcal{S}_{\tau,\tau^{\prime}}^{p}$ the complete space of $\mathbb{R}^d$-valued, càdlàg, and adapted processes   $(X_{t})_{t\geq 0}$ such that $\|X\|_{\mathcal{S}^{p}_{\tau,\tau^{\prime}}}:=\mathbb{E}\left(\sup_{t\in [\tau,\tau^{\prime}]}|X_{t}|^p\right)^{\frac{1}{p}}<\infty$.
		\item Let $\mathbb{A}$ be a compact metric space. By  $\mathcal{V}$ we denote the space of Radon measures on $[0,T]\times\mathbb{A}$ whose marginal on $[0,T]$ is the Lebesgue measure. It is endowed with the topology of stable convergence of measures; that is, convergence is required for all bounded measurable functions $\psi(t,a)$ such that, for each fixed $t\in[0,T]$, the function $\psi(t,\cdot)$ is continuous.
	\end{itemize}
	\section{The Skorokhod problem with mean reflection}\label{sk}
	We define the Skorokhod reflection problem with mean reflection as the following.
	\begin{definition}\label{sk_det}
		Let $\mathcal{D}$ be a convex closed domain, and let $Y$ be a càdlàg adapted process such that $Y_{0}\in \mathcal{D}$. Let $\phi:[0, T]\times \mathbb{R}^d\rightarrow \mathbb{R}^d $ be a measurable map. We say that a pair $(X,k)\in\mathbb{D}([0,T],\mathbb{R}^{2d})$ is a solution to the reflection problem associated with $Y$ if $X$ is càdlàg and adapted, and $k$ is a deterministic map with finite variation, and for every $t\in [0,T]$, the following conditions hold:
		\begin{itemize}	
			\item[(i)]  $X_t=Y_t+k_t$.
			\item[(ii)] $\mathbb{E}(\phi_t(X_t)) \in\mathcal{D}$.
			\item [(iii)] For every $z\in\mathcal{D}$, we have $$\int_0^t \langle \mathbb{E}(\phi_s(X_s)) - z,  \mathrm{d}k_s \rangle \leq 0.$$
		\end{itemize}
		This problem will be denoted by $RP_{\mathcal{D}}(\phi,Y)$.
	\end{definition}
	\begin{remark}\label{rq_iii}
		Note that point $(iii)$ in Definition \ref{sk_det} implies the existence of a unit vector $\eta$ such that $\mathrm{d}k_s=\eta_s \mathrm{d}|k|_s$, with $\eta_s\in\mathcal{N}_{\mathbb{E}(\phi_s(X_s))}$ whenever $\mathbb{E}(\phi_s(X_s))\in\partial\mathcal{D}$, and $|k|_t=\int_0^t \mathds{1}_{\left\{\mathbb{E}(\phi_s(X_s)) \in \partial\mathcal{D}_s\right\}} \mathrm{d}|k|_s$. This can be shown using the same reasoning as in the proof of Lemma 2.1 in \cite{petit}.
	\end{remark}
	We consider the following assumption on $\phi$.
	\begin{itemize}
		\item[$\left(\mathbf{H}\right)$:] There exist $v\in\mathbb{D}([0,T],\mathbb{R}^{d})$ and a linear, continuous, and bijective map $ l:\mathbb{R}^d\rightarrow\mathbb{R}^d$ such that, for every $(t,x)\in [0,T]\times\mathbb{R}^d$, $\phi_t(x)=l\circ l^{*}(x)+v_t$.
	\end{itemize}
	In the following, we consider the map $\Phi:[0,T]\times \mathbb{R}^d\rightarrow \mathbb{R}^d$ defined by 
	\begin{equation}
		\Phi_t(x):= l(x)+ v_t\quad\text{for all}\; (t,x)\in [0,T]\times\mathbb{R}^d.\label{Phi}
	\end{equation}
	\begin{remark}\label{rq_bilip}
		\begin{itemize}
			\item[i)] Since $l$ is linear, continuous, and bijective, The Banach isomorphism theorem implies that it is bi-Lipschitz. That is, there exist $\gamma_1,\gamma_2>0$ such that for every $x,x^{\prime}\in\mathbb{R}^{d}$, we have
			\begin{equation}
				\gamma_1\left|x-x^{\prime}\right|\leq   \left|l(x)-l(x^{\prime})\right| \leq \gamma_2\left|x-x^{\prime}\right|. \label{bilip}
			\end{equation}
			\item[ii)] For a given $t\in [0,T]$, it is clear that $x\mapsto\Phi_t(x)$ is continuous, bijective and satisfies \eqref{bilip}.
		\end{itemize}
	\end{remark}
	\begin{lemma}\label{lemma_det}
		Assume $\left(\mathbf{H}\right)$. Let $\mathcal{D}$ be a set in $\mathcal{C}$ and define $D_t:=\Phi^{-1}_{t}(\mathcal{D})$. Then, the following properties hold
		\begin{itemize}
			\item[i)] The family  of time-dependent sets $D=\{D_t,\;t\in [0,T]\}$ is contained in $\mathcal{C}$. Moreover, we have  $\mathring{D}_t= \Phi_{t}^{-1}(\mathring{\mathcal{D}})$, and  $\partial D_t=\Phi_{t}^{-1}(\partial \mathcal{D})$.
			\item [ii)]The mapping $t\mapsto D_t $ is càdlàg with respect to the Hausdorff metric.
			\item [iii)] If $x\in\mathbb{D}([0,T],\mathbb{R}^{d})$, then the mappings $t\mapsto\Phi_{t}(x_t)$ and $t\mapsto\Phi^{-1}_{t}(x_t)$ are càdlàg.
		\end{itemize}
	\end{lemma}
	\begin{proof}
		\begin{itemize}
			\item [$i)$] Using the affine nature and the continuity of the mapping $x\mapsto \Phi_{t}(x)$ for $t\in[0,T]$, it follows that the family of time-dependent sets $D=\{D_t,\;t\in [0,T]\}$ remains in $\mathcal{C}$. The other properties follow easily from the fact that $x\mapsto\Phi_{t}(x)$ is continuous and bijective.
			\item[$ii)$] Let $s,t\in [0,T]$, and $x\in D_s$, we have 
			\begin{align*}
				d(x,D_t)&= \inf\{ |x-y|\;,y\in D_t\}\\
				&\leq \frac{1}{\gamma_1}\inf\{ |\Phi_{t}(x)-\Phi_{t}(y)|\;,y\in D_t\}\\
				&\leq \frac{1}{\gamma_1}|\Phi_{t}(x)-\Phi_{s}(x)|+ \frac{1}{\gamma_1}\inf\{ |\Phi_{s}(x)-\Phi_{t}(y)|\;,y\in D_t\}\\
				&\leq \frac{1}{\gamma_1}|\Phi_{t}(x)-\Phi_{s}(x)|+\frac{\gamma_2}{\gamma_1} \inf\{ | \Phi_{t}^{-1}(\Phi_{s}(x))-y|\;,y\in D_t\}\\
				&= \frac{1}{\gamma_1}|v_{t}-v_{s}|+\frac{\gamma_2}{\gamma_1} d(\Phi_{t}^{-1}(\Phi_{s}(x)),D_t).
			\end{align*}
			In the second and fourth lines, we used the bi-Lipschitz property of $x\mapsto \Phi_t(x)$. Since $x\in D_s$, it follows that $\Phi_{s}(x)\in\mathcal{D}$. This implies that  $\Phi_{t}^{-1}(\Phi_{s}(x))\in D_t$, therefore $d(\Phi_{t}^{-1}(\Phi_{s}(x)),D_t)=0$. Consequently,
			$$\sup_{x\in D_s}d(x,D_t)\leq  \frac{1}{\gamma_1}|v_t-v_s|.$$
			Similarly, we obtain
			$$\sup_{x\in D_t}d(x,D_s)\leq\frac{1}{\gamma_1} |v_t-v_s|.$$
			It follows that 
			\begin{equation*}
				d_H\left(D_t, D_s\right)\leq \frac{1}{\gamma_1}|v_t-v_s|
			\end{equation*}
			from this inequality and using the fact that $v$ is right-continuous, it follows that $t\mapsto D_t$ is right continuous with respect to the Hausdorff metric. Since $(\mathcal{C},d_H)$ is a complete space, then $t\mapsto D_t$ is left-limited with respect to the Hausdorff metric.
			\item[$iii)$] Let $s,t\in [0,T]$, we have
			\begin{align*}
				|\Phi_{t}(x_t)-\Phi_{s}(x_s)|&\leq |\Phi_{t}(x_t)-\Phi_{s}(x_t)|+|\Phi_{s}(x_t)-\Phi_{s}(x_s)|\\
				&\leq|v_t-v_s|+\gamma_2|x_t-x_s|
			\end{align*}
			In the last line, we have used the inequality \eqref{bilip}. Since $x,v\in\mathbb{D}([0,T],\mathbb{R}^{d})$, the result follows.
			
			Similarly, we obtain that $t\mapsto\Phi^{-1}_{t}(x_t)$ is càdlàg.
		\end{itemize}
	\end{proof}
	
	\begin{theorem}\label{th1}
		Assume $\left(\mathbf{H}\right)$. Let $\mathcal{D}$ be an element of $\mathcal{C}$, and let $Y$ be an element of $\mathcal{U}$ such that $\mathbb{E}(\phi_0(Y_0))\in\mathcal{D}$. Then, there exists a unique solution to the Skorokhod problem $RP_{\mathcal{D}}(\phi,Y)$. Furthermore, there exists a constant $C>0$ such that 
		\begin{equation}
			|k|_T\leq C\left( \sup_{s\in [0,T]}|\mathbb{E}(Y_s)|^2+1\right). \label{estim_var_k}
		\end{equation}
	\end{theorem}
	\begin{proof}
		Let $D$ be the family of time-dependent sets defined in Lemma \ref{lemma_det}, and let $y$ be the function defined by $y_t:=\mathbb{E}(Y_t)$ for $t\in [0,T]$. Since $Y\in\mathcal{U}$, it follows that $y$ belongs to $\mathbb{D}([0,T],\mathbb{R}^d)$. We will establish the existence and uniqueness of solution of the Skorokhod problem $RP_{\mathcal{D}}(\phi,Y)$ by applying Theorem \ref{prop_A} in the appendix to the function $\overline{y}:t\mapsto l^{*}(y_t)$. To do this, we first verify the assumptions of Theorem \ref{prop_A}.
		\begin{itemize}
			\item \textbf{Step 1}: Find $a\in\mathbb{D}([0,T],\mathbb{R}^d)$ such that, for all $t\in [0,T]$, $a_t\in \mathring{D_t}$ and $\inf_{t\in[0,T]}d(a_t,\partial D_t)>0$.
			
			Let $\alpha$ be element of $\mathring{\mathcal{D}}$, noting that $\mathring{\mathcal{D}}\neq\emptyset$. For $t\in [0,T]$, define $a$ by setting $a_t=\Phi_{t}^{-1}(\alpha)$. By $(i)$ of Lemma \ref{lemma_det}, it follows that $a_t\in\mathring{D_t}$. Moreover by $(iii)$ of Lemma \ref{lemma_det}, the function $a$ is càdlàg. 
			
			We now compute the distance to the boundary
			\begin{align*}
				d(a_t,\partial D_t) &=\inf\{ |a_t-y|,\;y\in\partial D_t \}\\
				&\geq\frac{1}{\gamma_2} \inf\{ |\alpha-\Phi_{t}(y)|,\;y\in \Phi_{t}^{-1}(\partial\mathcal{D}) \}\\
				&\geq \frac{1}{\gamma_2} d(\alpha,\partial\mathcal{D})
			\end{align*}
			In the previous lines, we used the fact that $x\mapsto \Phi_t(x)$ is bi-Lipschitz. It follows that
			$$\inf_{t\in [0,T]}d(a_t,\partial D_t)\geq \frac{1}{\gamma_2}d(\alpha,\partial\mathcal{D})>0.$$
			\item \textbf{Step 2}: We establish the existence and uniqueness of solution of $RP_{\mathcal{D}}(\phi,Y)$.
			
			By Theorem \ref{prop_A}, there exists a unique solution $(\overline{x},\overline{k})\in\mathbb{D}([0,T],\mathbb{R}^{2d})$ such that $\overline{k}$ has finite variation, and for every $t\in [0,T]$,  $\overline{x}_t=\overline{y}_t+\overline{k}_t\in D_t$. Moreover, for every $\overline{z}\in\mathbb{D}([0,T],\mathbb{R}^{d})$ satisfying $\overline{z}_t\in D_t$, we have  
			$$\int_0^t \langle \overline{x}_{s} - \overline{z}_s,  \mathrm{d}\overline{k}_s \rangle \leq 0.$$
			Here, $\mathrm{d}\overline{k}_t=\overline{\eta}_t \mathrm{d}|\overline{k}|_t$ with $\overline{\eta}_t\in\mathcal{N}_{\overline{x}_t}$ whenever $\overline{x}_t\in\partial D_t$.
			
			Now, let $t\in [0,T]$. Define $X_t=Y_t+k_t$, where $$k_t:=(l^{*})^{-1}(\overline{k}_t)= \int_0^t (l^{*})^{-1}(\overline{\eta}_s) \mathrm{d}|\overline{k}|_s.$$ Since $l$ is linear. It follows that 
			\begin{align*}
				\mathbb{E}(\phi_t(X_t))&=\mathbb{E}(\phi_t(Y_t+k_t))\\
				&=\mathbb{E}(l\circ l^*(Y_t+k_t)+v_t)\\
				&=l\circ l^*(y_t+k_t)+v_t\\
				&=l(\overline{y}_t+\overline{k}_t)+v_t\\
				&=\Phi_{t}(\overline{x}_t)\in\mathcal{D}
			\end{align*}
			
			Next, we verify point $(iii)$ of Definition \ref{sk_det}. For $z\in\mathcal{D}$, we have
			\begin{align*}
				\int_{0}^{t} \langle \mathbb{E}(\phi_s(X_s))- z,  \mathrm{d}k_s \rangle&=  \int_{0}^{t} \langle \Phi_{s}(\overline{x}_s)- z, (l^{*})^{-1}(\overline{\eta}_s)  \rangle \mathrm{d}|\overline{k}|_s\\
				&= \int_{0}^{t} \langle l(\overline{x}_s)+v_s- z, (l^{*})^{-1}(\overline{\eta}_s)  \rangle \mathrm{d}|\overline{k}|_s\\
				&= \int_{0}^{t} \langle l(\overline{x}_s-l^{-1}(z-v_s)), (l^{*})^{-1}(\overline{\eta}_s)  \rangle \mathrm{d}|\overline{k}|_s\\
				&= \int_{0}^{t} \langle \overline{x}_s-l^{-1}(z-v_s), \overline{\eta}_s  \rangle \mathrm{d}|\overline{k}|_s
			\end{align*}
			In the last line, we used the adjoint property. Since $\Phi_s(l^{-1}(z-v_s))=z\in\mathcal{D}$, it follows that $l^{-1}(z-v_s)\in D_s$. As $(\overline{x},\overline{k})$ is a solution of $RP_{D}(\overline{y})$, we conclude
			$$\int_{0}^{t} \langle \overline{x}_s-l^{-1}(z-v_s), \overline{\eta}_s  \rangle \mathrm{d}|\overline{k}|_s\leq 0.$$
			Thus, $$ \int_{0}^{t} \langle \mathbb{E}(\phi_s(X_s))- z_s,  \mathrm{d}k_s \rangle\leq 0.$$
			
			Finally, we note that we can easily verify that if $(X,k)$ is a solution of $RP_{\mathcal{D}}(\phi,Y)$, then $(l^*(\mathbb{E}(X)),l^*(k))$ is a solution of $RP_{D}(l^*(y))$. This implies the uniqueness of the reflection problem $RP_{\mathcal{D}}(\phi,Y)$, completing the proof of its existence and uniqueness.
			\item \textbf{Step 3}: We prove inequality \eqref{estim_var_k}.
			We have 
			\begin{align*}
				|k|_T&=\sum_{i=1}^d|k^i|_T\\
				&=\sum_{i=1}^d \int_0^T|(l^{*})^{-1}_i(\overline{\eta}_s)| \mathrm{d}|\overline{k}|_s\\
				&\leq \frac{d}{\gamma_1} |\overline{k}|_T.
			\end{align*}
			where $(l^{*})^{-1}_i$ is the $i$-th component function of the linear operator $(l^{*})^{-1}$. By the inequality in Theorem \ref{prop_A}, we obtain
			\begin{align*}
				|k|_T&\leq\frac{dM}{\gamma_1}\left(\sup_{s\in [0,T]}|\overline{y}_s|^2+ \sup_{s\in [0,T]}|a_s|^2 +1\right)\\
				&\leq \frac{dM}{\gamma_1}\left( \sup_{s\in [0,T]}|l^{*}(y_s)|^2+ \sup_{s\in [0,T]}|\Phi_s^{-1}(\alpha)|^2 +1\right)
			\end{align*}
			Using the continuity of $l^*$ and the expression of $x\mapsto \Phi_t(x)$, we can find a constant $C$, depending on $\gamma_1,\gamma_2,\alpha,d$, and $\sup_{0\leq s\leq T}|v_s|$, such that
			$$|k|_T\leq C \left(\sup_{s\in [0,T]}|y_s|^2+1\right).$$
		\end{itemize}
	\end{proof}
	Now, we establish the following a priori estimate for the solutions of the problem $RP_{\mathcal{D}}(\phi,Y)$.
	\begin{proposition}\label{estime11}
		Assume $\left(\mathbf{H}\right)$. Let $Y,\tilde{Y}\in\mathcal{U}$, and let $\mathcal{D},\mathcal{\tilde{D}}\in\mathcal{C}$ be such that $\mathbb{E}(\phi_0(Y_0))\in\mathcal{D}$ and $\mathbb{E}(\phi_0(\tilde{Y}_0))\in\mathcal{\tilde{D}}$. Let $(X,k)$ and $(\tilde{X},\tilde{k})$ denote the solutions associated with $RP_{\mathcal{D}}(\phi,Y)$ and $RP_{\mathcal{\tilde{D}}}(\phi,\tilde{Y})$, respectively. Then, there exists a positive constant $C$ such that
		\begin{multline}
			|k_t-\tilde{k}_t|^2\leq \\C\left( |y_t-\tilde{y}_t|^2+ d_H(\mathcal{\tilde{D}},\mathcal{D})(|k|_t+|\tilde{k}|_t)+\int_{0}^{t}\langle y_t-y_s+\tilde{y}_s-\tilde{y}_t,l\circ l^{*}(\eta_s)\mathrm{d}|k|_s-l\circ l^{*}(\tilde{\eta}_s)\mathrm{d}|\tilde{k}|_s\rangle\right)\label{estime1}
		\end{multline}
		where $k_t=\displaystyle\int_0^t \eta_s\mathrm{d}|k|_s$ and $k_t=\displaystyle\int_0^t \tilde{\eta}_s\mathrm{d}|\tilde{k}|_s$, by Remark \ref{rq_iii}.
	\end{proposition}
	\begin{proof}
		For $t\in [0,T]$, we set $y_t=\mathbb{E}(Y_t)$, $\tilde{y}_t=\mathbb{E}(\tilde{Y}_t)$, $D=\{\Phi_t^{-1}(\mathcal{D}),\;t\in [0,T]\}$, and $\tilde{D}=\{\Phi_t^{-1}(\mathcal{\tilde{D}}),\;t\in [0,T]\}$. Also, $y^{1}_t=l^*(y_t)$ and $y^{2}_t=l^*(\tilde{y}_t)$. Let $(x^1,k^1)$ and $(x^2,k^2)$ be the solutions of $RP_D(y^1)$ and $RP_D(y^2)$, respectively. From the proof of Theorem \ref{th1}, $k=(l^{*})^{-1}(k^1)$ and $\tilde{k}=(l^{*})^{-1}(k^2)$. Therefore, we have
		\begin{multline*}
			|k_t-\tilde{k}_t|^2\leq \|(l^{*})^{-1}\|^2|k^1_t-k^2_t|^2\\
			\leq 2\|l^{-1}\|^2\left(|y_t^1-y_{t}^2|^2+ 2\sup_{s\in [0,t]} d_H(D_s,\tilde{D}_s)(|k^1|_t+|k^2|_t)+\int_{0}^{t}\langle y_t^1-y_s^1+y_s^2-y^2_t,\mathrm{d}(k_s^1-k_s^2) \rangle\right)
		\end{multline*}
		where in the second line, we used the estime \eqref{ineq_estim} of Proposition \ref{prop_B} in the Appendix.
		
		On the one hand, it is clear that $$|y_t^1-y_{t}^2|=|l^{*}(y_t)-l^*(\tilde{y}_t)|^2\leq \|l\|^2|y_t-\tilde{y}_t|^2\leq \gamma_2^2|y_t-\tilde{y}_t|^2.$$
		
		On the other hand, since $k^1_t=l^*(k_t)=\displaystyle\int_0^t l^*(\eta_s) \mathrm{d}|k|_s$ and $k^2_t=l^*(\tilde{k}_t)=\displaystyle\int_0^t l^*(\tilde{\eta}_s) \mathrm{d}|\tilde{k}|_s$, we obtain
		\begin{align*}
			\int_{0}^{t}\langle y_t^1-y_s^1+y_s^2-y^2_t,\mathrm{d}(k_s^1-k_s^2) \rangle&=\int_{0}^{t}\langle l^*(y_t)-l^*(y_s)+l^*(\tilde{y}_s)-l^*(\tilde{y}_t),l^*(\eta_s)\mathrm{d}|k|_s-l^*(\tilde{\eta}_s)\mathrm{d}|\tilde{k}|_s\rangle\\
			&=\int_{0}^{t}\langle y_t-y_s+\tilde{y}_s-\tilde{y}_t,l\circ l^{*}(\eta_s)\mathrm{d}|k|_s-l\circ l^{*}(\tilde{\eta}_s)\mathrm{d}|\tilde{k}|_s\rangle
		\end{align*}
		where, in the last line, we used the adjoint property.
		
		Now, we estimate $\sup_{s\in [0,t]} d_H(D_s,\tilde{D}_s)$. Let $s\in [0,T]$ and  $x\in\tilde{D}_s$, we use the Lipschitz property of $x\mapsto\Phi_t(x)$ we get
		$$d(x, D_s)\leq \gamma_2d(\Phi_s(x),\mathcal{D})$$
		since $\Phi_s(x)\in\tilde{\mathcal{D}}$, we have $d(x, D_s)\leq \gamma_2\sup_{x\in\mathcal{\tilde{D}}}d(x,\mathcal{D})$. Therefore, 
		$$d(x, D_s)\leq \gamma_2d_H(\mathcal{\tilde{D}},\mathcal{D}).$$
		In a similar way, for every $x\in D_s$, we obtain
		$$d(x, \tilde{D}_s)\leq \gamma_2d_H(\mathcal{\tilde{D}},\mathcal{D}).$$
		It follows that
		$$\sup_{s\in [0,t]} d_H(D_s,\tilde{D}_s)\leq \gamma_2 d_H(\mathcal{\tilde{D}},\mathcal{D})$$
	from Remark \ref{rq_iii}, the relationships between $k$ and $k^1$, and between $\tilde{k}$ and $k^2$, and since $l^*$ is linear and bijective, we can find a constant $C\geq 0$ such that
		$$|k_t-\tilde{k}_t|^2\leq C( |y_t-\tilde{y}_t|^2+ d_H(\mathcal{\tilde{D}},\mathcal{D})(|k|_t+|\tilde{k}|_t)+\int_{0}^{t}\langle y_t-y_s+\tilde{y}_s-\tilde{y}_t,l\circ l^{*}(\eta_s)\mathrm{d}|k|_s-l\circ l^{*}(\tilde{\eta}_s)\mathrm{d}|\tilde{k}|_s\rangle.$$

	\end{proof}
	\begin{corollary}
		Under the same assumptions as in Theorem \ref{th1}, except that $\mathcal{D}$ is unbounded, the existence and uniqueness of the reflection problem $RP_{\mathcal{D}}(\phi,Y)$ still hold, along with the inequality \eqref{estim_var_k}.
	\end{corollary}
	\begin{proof}
		We set $N=\lceil C( \sup_{s\in [0,T]}|\mathbb{E}(Y_s)|^2+1) \rceil$, where $C$ is the constant in the estimate \eqref{estim_var_k}. We consider the following sequence of elements in $\mathcal{C}$,
		$$\mathcal{D}^n= \begin{cases}\mathcal{D}\cap B(0,n), &\text{if}\; n\leq N,  \\ 
			\mathcal{D}\cap B(0,N), & \text{otherwise} .\end{cases}$$
		Let $(X^n,k^n)$ be the solutions of the reflection problem $RP_{\mathcal{D}^n}(\phi,Y)$. By the definition of $\mathcal{D}^n$ and the estimates \eqref{estim_var_k} and \eqref{estime1}, it is clear that $k^n$ converges uniformly and stationary to a limit $k$, which also has finite variation. Setting $X=Y+k$, it follows that $\mathbb{E}(\phi_t(X_t))\in\mathcal{D}$, and by Lemma 3.9 in \cite{ouknine20}, we conclude that statement $(iii)$ of the Definition \ref{sk_det} is satisfied.
	\end{proof}
	\begin{proposition}\label{estim_quad}
		Assume $\left(\mathbf{H}\right)$. Let $Y$ and $\tilde{Y}$ be two semimartingales with the following decomposition: $Y=M+V$ and $\tilde{Y}=\tilde{M}+\tilde{V}$, such that $\mathbb{E}(\phi_0(Y_0)),\mathbb{E}(\phi_0(\tilde{Y}_0))\in\mathcal{D}$. For $p\geq 1$, assume that $M$ and $\tilde{M}$ are $\mathcal{S}^{2p}_{0,T}$, and that $V$ and $\tilde{V}$ are càdlàg adapted processes with finite variation, such that $|V|$ and $|\tilde{V}|$ have a finite $2p$-th moment. Let $(X,k)$ and  $(\tilde{X},\tilde{k})$ be the solutions associated with $RP_{\mathcal{D}}(\phi,Y)$ and $RP_{\mathcal{D}}(\phi,\tilde{Y})$, respectively. Then, there exists a constant $C_p$  such that for all stopping times $\tau$ in $[0,T]$, we have
		\begin{equation}
			\mathbb{E}\left(\sup_{0\leq t\leq \tau}|X_{t}-\tilde{X}_{t}|^{2p}\right)+\sup_{0\leq t\leq \tau}|k_{t}-\tilde{k}_{t}|^{2p}\leq C_p\mathbb{E}\left([M-\tilde{M}]_{\tau}^p+|V-\tilde{V}|_{\tau}^{2p}\right). \label{ineq_estim_quad}
		\end{equation}
	\end{proposition}
	\begin{proof}
		Using Proposition \ref{estime11}, we have  
		\begin{align*}
			\sup_{0\leq t\leq \tau}|X_{t}-\tilde{X}_{t}|^{2p}&\leq C_p\left(\sup_{0\leq t\leq \tau}|Y_{t}-\tilde{Y}_{t}|^{2p}+\sup_{0\leq t\leq \tau}|y_{t}-\tilde{y}_{t}|^{2p}\right)\\
			&+C_p\sup_{t\in [0,\tau]}\left|\int_{0}^{t}\langle y_{t}-y_s+\tilde{y}_s-\tilde{y}_{t},l\circ l^{*}(\eta_s)\mathrm{d}|k|_s-l\circ l^{*}(\tilde{\eta}_s)\mathrm{d}|\tilde{k}|_s\rangle\right|^p.
		\end{align*}     
		where $C_p=C3^{p-1}$, and $C$ is the constant in the estimate \eqref{estime1}. In what follows, $C$ and $C_p$ are constants that may vary from line to line.
		
		We set $A=l\circ l^*(k)$ and $\tilde{A}=l\circ l^*(\tilde{k})$. Since $A$ and $\tilde{A}$ are deterministic, for $t\in [0,\tau]$, we have 
		$$\int_{0}^{t}\langle y_t-y_s+\tilde{y}_s-\tilde{y}_t,l\circ l^{*}(\eta_s)\mathrm{d}|k|_s-l\circ l^{*}(\tilde{\eta}_s)\mathrm{d}|\tilde{k}|_s\rangle
		=\mathbb{E}\left(\int_{0}^{t}\langle Y_t-Y_s+\tilde{Y}_s-\tilde{Y}_t,\mathrm{d}A_s-\mathrm{d}\tilde{A}_s\rangle\right)$$
		Applying It\^o's formula, we get
		\begin{align*}
			\mathbb{E}\left(\int_{0}^{t}\langle Y_t-Y_s+\tilde{Y}_s-\tilde{Y}_t,\mathrm{d}A_s-\mathrm{d}\tilde{A}_s\rangle\right)
			&=\mathbb{E}\left(\int_{0}^{t}\langle A_{s^-}-\tilde{A}_{s^-},\mathrm{d}Y_s-\mathrm{d}\tilde{Y}_s\rangle\right)\\
			&= \mathbb{E}\left(\int_{0}^{t}\langle l\circ l^*(X_{s^-}-\tilde{X}_{s^-}),\mathrm{d}(Y_s-\tilde{Y}_s)\rangle\right)\\
			&- \mathbb{E}\left(\int_{0}^{t}\langle l\circ l^*(Y_{s^-}-\tilde{Y}_{s^-}),\mathrm{d}(Y_s-\tilde{Y}_s)\rangle\right)\\
			&\leq\sum_{i=1}^d \mathbb{E}\left(\sup_{t\in[0,\tau]}|\int_{0}^{t}l^i\circ l^*(X_{s^-}-\tilde{X}_{s^-})\mathrm{d}(M_s^i-\tilde{M}_s^i)| \right)\\
			&+ \sum_{i=1}^d \mathbb{E}\left(\sup_{t\in[0,\tau]}|\int_{0}^{t}l^i\circ l^*(Y_{s^-}-\tilde{Y}_{s^-})\mathrm{d}(M_s^i-\tilde{M}_s^i)| \right)\\
			&+\sum_{i=1}^d \mathbb{E}\left(\sup_{t\in[0,\tau]}|\int_{0}^{t}l^i\circ l^*(X_{s^-}-\tilde{X}_{s^-})\mathrm{d}(V_s^i-\tilde{V}_s^i)| \right)\\
			&+ \sum_{i=1}^d \mathbb{E}\left(\sup_{t\in[0,\tau]}|\int_{0}^{t}l^i\circ l^*(Y_{s^-}-\tilde{Y}_{s^-})\mathrm{d}(V_s^i-\tilde{V}_s^i)| \right)
		\end{align*}
		
		We apply Burkholder's Davis Gundy inequality, we obtain
		\begin{align*}
			\mathbb{E}\left(\int_{0}^{t}\langle Y_t-Y_s+\tilde{Y}_s-\tilde{Y}_t,\mathrm{d}A_s-\mathrm{d}\tilde{A}_s\rangle\right)
			&\leq \frac{1}{2^{\frac{1}{p}}}\frac{1}{4^{\frac{p-1}{p}}} \mathbb{E}(\sup_{0\leq t\leq \tau}|X_{t}-\tilde{X}_{t}|^{2})\\
			&+ C_p\left( \mathbb{E}(\sup_{0\leq t\leq \tau}|Y_{t}-\tilde{Y}_{t}|^{2})+ \mathbb{E}([M-\tilde{M}]_{\tau})+\mathbb{E}(|V-\tilde{V}|_{\tau}^2)\right)
		\end{align*}
		therefore    \begin{align*}
			\left|\mathbb{E}\left(\int_{0}^{t}\langle Y_t-Y_s+\tilde{Y}_s-\tilde{Y}_t,\mathrm{d}A_s-\mathrm{d}\tilde{A}_s\rangle\right)\right|^p
			&\leq \frac{1}{2} \mathbb{E}(\sup_{0\leq t\leq \tau}|X_{t}-\tilde{X}_{t}|^{2p})\\
			&+ C_p\left( \mathbb{E}(\sup_{0\leq t\leq \tau}|Y_{t}-\tilde{Y}_{t}|^{2p})+ \mathbb{E}([M-\tilde{M}]_{\tau}^p+\mathbb{E}(|V-\tilde{V}|_{\tau}^{2p})\right)
		\end{align*}
		it follows from the decompositions of $Y$ and $\tilde{Y}$, together with an application of the Burkholder Davis Gundy inequality, that
		\begin{equation*}
			\mathbb{E}\left(\sup_{0\leq t\leq \tau}|X_{t}-\tilde{X}_{t}|^{2p}\right)\leq C_p\mathbb{E}\left([M-\tilde{M}]_{\tau}^p+ |V-\tilde{V}|_{\tau}^{2p}\right).
		\end{equation*}
		Since $X=Y+k$, the estimate for $\sup_{0\leq t\leq \tau}|k_{t}-\tilde{k}_{t}|^{2p}$ follows as well. 
	\end{proof}
	\begin{corollary}\label{cor_1}
		Under the same conditions as in Proposition \ref{estim_quad}, there exists a constant $C$ such that for every $0\leq t<q\leq T$, we have:
		\begin{equation}
			\mathbb{E}\left(\sup_{t\leq s\leq q }|X_{t}-X_s|^{2}\right)+\sup_{t\leq s\leq q}|k_{t}-k_{s}|^{2}\leq C\mathbb{E}\left([M]_{q}-[M]_t+ (|V|_{q}-|V|_{t})^2\right). \label{estim_vari}
		\end{equation}
	\end{corollary}
	\begin{proof}
		Consider $\tilde{Y}:=Y^{t}$. From the uniqueness of the solution, we have $\tilde{X}=X^t$ and $\tilde{k}=k^{t}$. The inequality holds by applying inequality \eqref{ineq_estim_quad} to $(X,k)$ and $(\tilde{X},\tilde{k})$.
	\end{proof}
	\section{McKean-Vlasov SDEs with mean reflection}\label{sdes}
	\subsection{Existence and uniqueness of strong solution}
	In this section, we are interested in the following system.
	\begin{definition}\label{def_EDS}
		Let $\mathcal{D}$ be a closed convex domain, and $X_{0}$ be an $\mathcal{F}_{0}$-measurable random variable. The maps
		$(b, \sigma):[0, T]\times \Omega \times \mathbb{R}^d \times \mathcal{P}\left(\mathbb{R}^d\right) \rightarrow \mathbb{R}^d \times \mathbb{R}^{d \times m}$, $\beta:[0, T] \times \Omega \times \mathbb{R}^d  \times \mathcal{P}\left(\mathbb{R}^d\right) \times \mathbb{R}^d \backslash\{0\} \rightarrow \mathbb{R}^d$, and $\phi:[0, T]\times \mathbb{R}^d\rightarrow\mathbb{R}^d$ are given.
		A couple $(X,k)$ is said to be a solution to the McKean-Vlasov SDE with mean reflection, which we denote by $E_{\mathcal{D}}(\phi, b,\sigma,\beta)$, if $(X,k)=RP_{\mathcal{D}}(\phi,Y)$, where $Y$ is given by the following:
		\begin{equation*}
			Y= X_{0}+ \int_{0}^{\cdot} b\left(s, X_s, \mathcal{L}_{X_s}\right) \mathrm{d} s+ \int_{0}^{\cdot} \sigma\left(s, X_s, \mathcal{L}_{X_s}\right) \mathrm{d} B_s+\int_{0}^{\cdot}\int_{\mathbb{R}^d \backslash\{0\}}\beta\left(s, X_{s^-}, \mathcal{L}_{X_{s^-}},z\right)  \tilde{N}(\mathrm{d} s, \mathrm{d} z).
		\end{equation*}
	\end{definition}
	We consider the following assumptions:
	\begin{itemize}
		\item[$\left(\mathbf{A_1^p}\right)$]: $X_{0}$ is an $\mathcal{F}_{0}$-measurable random variable with $\mathbb{E}(\phi(X_{0}))\in \mathcal{D}$ and $\mathbb{E}(|X_0|^p)<\infty$.
		\item[$\left(\mathbf{A_2^p}\right)$]: For fixed $x\in\mathbb{R}^d$ and $\mu\in \mathcal{P}(\mathbb{R}^d)$, the processes $b(\cdot,\cdot,x,\mu)$ and $\sigma(\cdot,\cdot,x,\mu)$ are  elements of $\mathbb{H}^{p,d}$ and $\mathbb{H}^{p,dm}$, respectively. In addition, for $(\omega,t)$ fixed, there exists $\gamma>0$ such that $$\quad\left|(b, \sigma)(t, \omega, x, \mu)-(b, \sigma)\left(t, \omega, x^{\prime}, \mu^{\prime}\right)\right| \leq \gamma\left(|x-x^{\prime}|+W_p(\mu, \mu^{\prime})\right).$$
		\item[ $\left(\mathbf{A_3^p}\right)$]: $\beta$ is $\mathscr{P} \otimes \mathscr{B}_{\mathbb{R}^d}\otimes\mathcal{P}_{2}(\mathbb{R}^d) \otimes \mathcal{B}_{\mathbb{R}^d \backslash\{0\}}$-measurable. For every $(t,\omega)$, we have
		$$\left(\int_{\mathbb{R}^d \backslash\{0\}}\left|\beta(s,\omega,x,\mu, z)-\beta\left(s,\omega,x^{\prime},\mu^{\prime}, z\right)\right|^2\lambda(\mathrm{d} z)\right)^\frac{1}{2} \leq \gamma \left(|x-x^{\prime}|+W_2(\mu, \mu^{\prime})\right),$$
		$$\int_{\mathbb{R}^d \backslash\{0\}} |\beta(s,\omega,x,\mu, z)|^p\lambda(\mathrm{d} z)\leq \gamma\left(1+|x|^p\right).$$
	\end{itemize}
	\begin{theorem}\label{theorem_1}
		Under assumptions $\left(\mathbf{H}\right)$ and $\left(\mathbf{A_1^2}\right)$-$\left(\mathbf{A_3^2}\right)$, for every $\mu \in\mathcal{P}_2(\mathbb{D}([0,T],\mathbb{R}^d))$, there exists a unique solution $(X^{\mu},K^{\mu})=RP_{\mathcal{D}}(\phi,Y^{\mu})$, where $X^{\mu}\in\mathcal{S}^{2}_{0,T}$, and $Y^{\mu}$ is given by: 
		\begin{equation}
			Y_{\cdot}^{\mu}= X_{0}+ \int_{0}^{\cdot} b\left(s, X_s^{\mu}, \mu_s\right) \mathrm{d} s+ \int_{0}^{\cdot} \sigma\left(s, X_s^{\mu}, \mu_s\right) \mathrm{d} B_s+\int_{0}^{\cdot}\int_{\mathbb{R}^d\backslash\{0\}}\beta\left(s,  X_{s^-}^{\mu}, \mu_{s^-},z\right)  \tilde{N}(\mathrm{d} s, \mathrm{d} z).\label{Eq1}
		\end{equation}
	\end{theorem}
	\begin{proof}
		Let $\alpha \in ]0,1[$, we consider the following stopping time:
		\begin{equation*}
			\tau:=\inf\{t>0 \;:\;  C\gamma^2(2t+t^2)\geq \alpha\}\land T,
		\end{equation*}
		where $C$ is the constant from Proposition \ref{estim_quad}, and $ \tau=\infty\; \text{with} \;\inf(\emptyset)$.
		
		Now, we define the mapping $\varphi \colon\mathcal{S}^{2}_{0,\tau}\to\mathcal{S}_{0,\tau}^{2}$ that associates to $X\in\mathcal{S}_{0,\tau}^{2}$ the first coordinate of the solution to the reflection problem $RP_{\mathcal{D}}\Big(Y^\mu_{.\land\tau} \Big)$, This mapping is well defined by Theorem \ref{th1}.
		
		For every $X\in \mathcal{S}^{2}_{0,\tau}$, we have $\varphi(X)\in \mathcal{S}^{2}_{0,\tau}$. Indeed, 
		using assumptions $\left(\mathbf{A}_2^2\right)$ and $\left(\mathbf{A}_3^2\right)$, and applying the Burkholder-Davis-Gundy inequality, there exists a constant $C_{T,\gamma,\mu}$ such that
		\begin{align*}
			\| Y^{\mu}\|_{\mathcal{S}^{2}_{0,\tau}}^2&\leq C_{T,\gamma,\mu} \mathbb{E}\left(|X_0|^2+\int_{0}^{\tau} |X_s|^2 \mathrm{d} s+\int_{0}^{\tau} |b(s,0,\mu_s)|^2\mathrm{d}s+\int_{0}^{\tau} |\sigma(s,0,\mu_s)|^2\mathrm{d}s\right)\\
			&+C_{T,\gamma,\mu}\mathbb{E}\left(1+\sup_{0\leq t\leq \tau}|X_{t}|^{2}\right)\\
			&\leq C_{T,\gamma,\mu}\left(\mathbb{E}\left(|X_0|^2 \right) +\|X\|_{\mathcal{S}^{2}_{0,\tau}}^2+1\right).
		\end{align*}
		Since $X\in\mathcal{S}^{2}_{0,\tau}$ and using assumptions $\left(\mathbf{A}_2^2\right)$, $\left(\mathbf{A}_3^2\right)$, along with the definition of the time $\tau$, we can conclude from Proposition \ref{estim_quad}, that  $\|\varphi(X^{\mu})\|_{\mathcal{S}^{2}_{0,\tau}}<\infty$.
		
		Now, we show that $\varphi$ is a contraction. Let $X,\tilde{X}\in \mathcal{S}^{2}_{0,\tau}$. By Proposition \ref{estim_quad}, we have
		\begin{align*}
			\| \varphi(X)-\varphi(\tilde{X}) \|_{\mathcal{S}^{2}_{0,\tau}}^2&\leq \mathbb{E}\left(\int_{0}^{\tau} |\sigma\left(s, X_s, \mu_s\right)-\sigma(s, \tilde{X}_s, \mu_s)|^2 \mathrm{d} s\right)\\
			&+C\mathbb{E}\left( \int_{0}^{\tau} \int_{\mathbb{R}^d\backslash\{0\}}|\beta(s, X_{s^-}, \mu_{s^-},z)-\beta(s, \tilde{X}_s, \mu_{s^-},z)|^2 \lambda(\mathrm{d} z)\mathrm{d}s\right)\\
			&+C\mathbb{E}\left(\tau \int_{0}^{\tau} |b(s, X_s, \mu_s)-b(s, \tilde{X}_s, \mu_s)|^2 \mathrm{d} s\right)\\
			& \leq C\gamma^2 \mathbb{E}\left((2\tau+\tau^2) \sup_{0\leq t\leq \tau}|X_{t}-\tilde{X}_{t}|^{2}\right)\\
			&\leq \alpha \|X-\tilde{X}\|_{\mathcal{S}^{2}_{0,\tau}}^2.\\
		\end{align*}
		Thus, by the Banach fixed-point theorem, we obtain the existence and uniqueness of the solution on $[0,\tau]$. By iterating, we extend the result on $[0,T]$.
		
		Finally, we show that $X^{\mu}\in\mathcal{S}^{2}_{0,T}$. Using assumptions $\left(\mathbf{A}_2\right)$ and $\left(\mathbf{A}_3\right)$, along with inequality \eqref{ineq_estim_quad}, we obtain
		\begin{align*}
			\| X^{\mu}\|_{\mathcal{S}^{2}_{0,T}}^2&\leq C_{T,\gamma} \mathbb{E}\left(|X_0|^2+\int_{0}^{T} (1+\sup_{0 \leq s \leq u}|X_u|^2) \mathrm{d} s+\int_{0}^{T} |b(s,0,\mu_s)|^2\mathrm{d}s+\int_{0}^{T} |\sigma(s,0,\mu_s)|^2\mathrm{d}s\right).\\
		\end{align*}
		Applying Gronwall's Lemma, and using assumptions $\left(\mathbf{A_1^2}\right)$, $\left(\mathbf{A_2^2}\right)$ and $\left(\mathbf{A_3^2}\right)$, we conclude that $X^{\mu}$ is and element of $\mathcal{S}^{2}_{0,T}$.
		This completes the proof.
	\end{proof}
	\begin{theorem}\label{Main_th}
		Assume $\left(\mathbf{H}\right)$, and $\left(\mathbf{A_1^2}\right)$-$\left(\mathbf{A_3^2}\right)$. Then, there exists a unique solution to the system $E_{\mathcal{D}}(\phi, b,\sigma,\beta)$.
	\end{theorem}
	\begin{proof}
		By Theorem \ref{theorem_1}, for a fixed $\mu \in\mathcal{P}_2(\mathbb{D}([0,T],\mathbb{R}^d))$, there exists a unique solution $(X^{\mu},k^{\mu})$ to the problem $RP_{\mathcal{D}}(\phi,Y^{\mu})$ such that $X^{\mu}\in\mathcal{S}^{2}_{0,T}$ and $Y^{\mu}$ is given by \eqref{Eq1}.
		For $t\in [0,T]$, we define the mapping $\psi\colon\mathcal{P}_2(\mathbb{D}([0,t],\mathbb{R}^d))\to\mathcal{P}_2(\mathbb{D}([0,t],\mathbb{R}^d))$, which associates $\mu\in\mathcal{P}_2(\mathbb{D}([0,t],\mathbb{R}^d))$ with $\psi(\mu)=\mathcal{L}_{X^{\mu}}$. We show that $\psi$ is a contraction.
		
		Consider $\mu,\mu^{\prime}\in\mathbb{D}([0,t],\mathbb{R}^d)$. By Proposition \ref{estim_quad} and under assumptions $\left(\mathbf{A}_2^2\right)$ and $\left(\mathbf{A}_3^2\right)$, for every $t\in [0,T]$, we have
		\begin{align*}
			\mathbb{E}\left(\sup_{0\leq s\leq t}|X_{s}^{\mu}-X_{s}^{\mu^{\prime}}|^{2}\right)&\leq C\gamma^2 (2+t)\left(\int_{0}^{t}\mathbb{E}(\sup_{0\leq u\leq s}|X_{u}^{\mu}-X_{u}^{\mu^{\prime}}|^{2})+W_2(\mu_s, \mu^{\prime}_s)^2\mathrm{d}s\right)\\
			& + C\gamma^2 (2+t)\int_{0}^{t} W_2(\mu_{s^-}, \mu_{s^-}^{\prime})^2 \mathrm{d} s \\
			&\leq 2C\gamma^2 (2+t)\left(\int_{0}^{t}\mathbb{E}(\sup_{0\leq u\leq s}|X_{u}^{\mu}-X_{u}^{\mu^{\prime}}|^{2}) \mathrm{d}s+ tW_{\{t,2\}}(\mu, \mu^{\prime})^2 \right).
		\end{align*}
		In the final line, we have used the fact that $W_2(\mu_s, \mu^{\prime}_s)\lor W_2(\mu_{s^-}, \mu^{\prime}_{s^-})\leq W_{\{s,2\}}(\mu, \mu^{\prime})$. By applying Gronwall’s inequality and considering the definition of the Wasserstein metric, we obtain
		\begin{equation*}
			W_{\{t,2\}}\left(\psi(\mu), \psi(\mu^{\prime})\right)^2\leq 2C\gamma^2 t(2+t) C_{T,\gamma} W_{\{t,2\}}(\mu, \mu^{\prime})^2.
		\end{equation*}
		Here, $C_{T,\gamma}$ is a constant depending on $T$ and $\gamma$. Therefore, as in the proof of Theorem \ref{theorem_1}. we can find a positive time $\tau$ such that for all $t\leq \tau$, $\psi$ is a contraction. The proof follows similarly to that of Theorem \ref{theorem_1}.
	\end{proof}
	\begin{proposition}\label{estim-VAR}
		Assume $\left(\mathbf{H}\right)$, and $\left(\mathbf{A_1^2}\right)$-$\left(\mathbf{A_3^2}\right)$. Let $(X,k)$ be the solution to the system $E_{\mathcal{D}}(\phi, b,\sigma,\beta)$. For $s,t\in [0,T]$, there exists a càdlàg function $c$ with $c(0)=0$ such that
		\begin{equation*}
			\mathbb{E}(\sup_{s\leq r\leq t}|X_s-X_r|^2)+\sup_{s\leq r\leq t}|k_{s}-k_{r}|^{2}\leq c(t-s).
		\end{equation*}
	\end{proposition}
	\begin{proof}
		Let $0\leq s\leq t\leq T$. From Corollary \ref{cor_1}, we have
		\begin{multline*}
			\mathbb{E}\left(\sup_{s\leq r\leq t}|X_s-X_r|^2\right)+\sup_{s\leq r\leq t}|k_{s}-k_{r}|^{2}\leq \\C\mathbb{E}\left(\int_s^t |\sigma(u,X_u,\mathcal{L}_{X_u})|^2\mathrm{d}u+ \int_s^t\int_{\mathbb{R}^d\backslash{\{0\}}} |\beta(u,X_u,\mathcal{L}_{X_u},z)|^2\lambda(\mathrm{d}z)\mathrm{d}u+\int_s^t |b(u,X_u,\mathcal{L}_{X_u})|^2\mathrm{d}u\right)    
		\end{multline*}
		using the Lipchitz assumptions on the coefficients, we obtain
		\begin{multline*}
			\mathbb{E}\left(\sup_{s\leq r\leq t}|X_s-X_r|^2\right)+\sup_{s\leq r\leq t}|k_{s}-k_{r}|^{2}\leq\\ C(t-s)\left(1+\mathbb{E}(\sup_{0\leq s\leq T }|X_s|^2)+\int_0^T \mathbb{E}(\sigma((t-s)u,0,\delta_0)^2)\mathrm{d}u+\int_0^T \mathbb{E}(b((t-s)u,0,\delta_0)^2)\mathrm{d}u\right)
		\end{multline*}
		Thus, $c$ is given by $$c(x)= C|x|\left(1+\mathbb{E}(\sup_{0\leq s\leq T }|X_s|^2)+\int_0^T \mathbb{E}(\sigma(|x|u,0,\delta_0)^2)\mathrm{d}u+\int_0^T \mathbb{E}(b(|x|u,0,\delta_0)^2)\mathrm{d}u\right)$$
		which is well defined by assumption $\left(\mathbf{A_2^2}\right)$ and the fact that $X\in\mathcal{S}_{0,T}^{2}$.
	\end{proof}
	\begin{corollary}
		Under the same conditions as in Proposition \ref{estim-VAR}, the measure $\mathrm{d}k$ is absolutely continuous with respect to the Lebesgue measure. 
	\end{corollary}
	\begin{proof}
		This result follows directly from the estimate in Proposition \ref{estim-VAR}.  
	\end{proof}
	\subsection{Particle approximation of mean-reflected McKean–Vlasov SDEs}
	Let $N\geq 1$ and $i\in\{1,\dots,N\}$. For $t\in [0,T]$, we consider the following $N$-particle system: \\
	
	$\textbf{(S)} \left\{\begin{aligned}
		&X_t^{i,N}= X_{0}^i+ \int_{0}^{t} b\left(s, X_s^{i,N}, \frac{1}{N}\sum_{j=1}^N\delta_{X_s^{j,N}}\right) \mathrm{d}s+ \int_{0}^{t} \sigma\left(s, X_s^{i,N},\frac{1}{N}\sum_{j=1}^N\delta_{X_s^{j,N}}\right) \mathrm{d} B_s^i \\
		& \quad\quad+\int_{0}^{t}\int_{\mathbb{R}^d\backslash\{0\}}\beta\left(s, X_{s^-}^{i,N}, \frac{1}{N}\sum_{j=1}^N\delta_{X_{s^-}^{j,N}},z\right) \tilde{N}^i(\mathrm{d} s, \mathrm{d} z)+ K_t^{N}, \\
		&
		\frac{1}{N}\sum_{i=1}^N\phi_t(X_t^{i,N})\in \mathcal{D}, \\
		& \text{for every} \;z\in\mathcal{D},\; \int_0^t \langle \frac{1}{N}\sum_{i=1}^N\phi_s(X_s^{i,N}) - z,  \mathrm{d}K_s^{N} \rangle \leq 0.
	\end{aligned}\right.$
	
	Where, $(X_{0}^i,B^i,\tilde{N}^i)$ are sequence of independent copies of $(X_0,B,\tilde{N})$.
	\begin{proposition}
		Assume $\left(\mathbf{H}\right)$, and $\left(\mathbf{A_1^2}\right)$-$\left(\mathbf{A_3^2}\right)$. There exists a unique solution for the system $\textbf{\upshape (S)}$.
	\end{proposition}
	\begin{proof}
		Let $\{D_t\}_{t\in [0,T]}$ be the family of time-dependent sets defined by $$D_t=\left\{(x^1,\dots,x^N)\in\mathbb{R}^{Nd}:\; \Phi_t(x^i)\in\mathcal{D},\;\text{for all}\;i\in{1,\dots,N}\right\},\;\;t\in [0,T]$$
		where $\Phi$ is given in \eqref{Phi}.
		
		For $(t,x^1,\dots,x^N,z)\in[0,T]\times\mathbb{R}^{Nd}\times\mathbb{R}^d\backslash{\{0\}}$, let $\tilde{b}$, $\tilde{\sigma}$, and $\tilde{\beta}$ be the mappings whose components are given by  
		$$\tilde{b}(t,x) := \left( b\left(t, x^i, \frac{1}{N} \sum_{j=1}^N \delta_{x^j} \right) \right)_{i=1,\dots,N}\in\mathbb{R}^{Nd},\;\; \tilde{\sigma}(t,x) := \left( \sigma\left(t, x^i, \frac{1}{N} \sum_{j=1}^N \delta_{x^j} \right) \right)_{i=1,\dots,N}\in\mathbb{R}^{d\times Nm},$$
		and $$
		\tilde{\beta}(t,x,z):= \left( \beta\left(t, x^i, \frac{1}{N} \sum_{j=1}^N \delta_{x^j},z \right) \right)_{i=1,\dots,N}\in\mathbb{R}^{Nd}.$$ 
		Let $X$ be an $\mathbb{R}^{Nd}$-valued process. Define \begin{equation*}
			Y^{X,N}_{\cdot}= X_{0}+ \int_{0}^{\cdot} \tilde{b}(s, X_s) \mathrm{d} s+ \int_{0}^{\cdot}\tilde{\sigma}(s, X_s) \mathrm{d} B_s+\int_{0}^{\cdot}\int_{\mathbb{R}^d\backslash\{0\}}\tilde{\beta}(s,X_{s^-},z) \tilde{N}(\mathrm{d} s, \mathrm{d} z).
		\end{equation*}
		where $X_0:=(X_0^1,\dots,X_0^N)\in\mathbb{R}^{Nd} $, $B:=(B^0,\dots,B^N)\in\mathbb{R}^{Nm}$, and $\tilde{N}:=(\tilde{N}^0,\dots,\tilde{N}^N)\in\mathbb{R}^{Nd}$.
		Define $$\overline{Y}^{X,N}=\left(l^*\left( \frac{1}{N}\sum_{j=1}^NY^{X,N,j}\right),\dots,l^*\left(\frac{1}{N}\sum_{j=1}^NY^{X,N,j}\right)\right)\in\mathbb{R}^{Nd}.$$
		By repeating the arguments from Step 2 of the proof of Theorem \ref{th1}, we obtain an adapted solution $(\overline{X},\overline{K})$ to the problem $RP_{\{D_t\}_{t\in [0,T]}}(\overline{Y}^{X,N})$. 
		
		Define the mapping $\varphi$ by
		$$\varphi(X):=Y^{X,N}+K^N,\;\; \text{where}\;\; K^N:=\left((l^{*})^{-1}(\overline{K}^1),\dots,(l^{*})^{-1}(\overline{K}^N)\right),$$
		we have that the $K^{i,N}=K^{j,N}$ for all $i\neq j$. Thus, we can  verify similarly as to Step 2 of the proof of Theorem \ref{th1} that $(\varphi(X),K^N)$ satisfies the second and third statements of system $\textbf{\upshape (S)}$. 
		
		Using the assumptions on $b$, $\sigma$, and $\beta$, and noting that for all $u=(u^1,\dots,u^N),\;v=(v_1,\dots,v^N)\in\mathbb{R}^{Nd}$, we have
		$$W_2\left(\frac{1}{N} \sum_{i=1}^N \delta_{u_i}, \frac{1}{N} \sum_{i=1}^N \delta_{v_i}\right) \leq\frac{1}{\sqrt{N}}|u-v|,$$ 
		we can repeat the arguments from the proof of Theorem \ref{theorem_1} to establish the existence and uniqueness of $X^N$ such that $\varphi(X^N)=X^N$. Therefore, the first statement of $\textbf{\upshape (S)}$ holds. Since all components of $K^N$ are identical, we may, without confusion, write $K^{i,N}=K^N$ for all $i$ in system $\textbf{\upshape (S)}$.
	\end{proof}
	\begin{theorem}
		Let $p\geq 4$. Assume $\left(\mathbf{H}\right)$, and $\left(\mathbf{A_1^p}\right)$-$\left(\mathbf{A_3^p}\right)$. Let $(X^{i},k)$ be the solution of $RP_{\mathcal{D}}(\phi,Y^i)$, adapted to the complete filtration generated by $X_0^i$, $B^i$, and $N^i$, where $Y^i$ is given by
		$$Y_t^i=X_{0}^i+\int_{0}^{t} b(s,X_s^i, \mathcal{L}_{X_s^i})\mathrm{d}s+ \int_{0}^{t} \sigma(s, X_s^i,\mathcal{L}_{X_s^i})\mathrm{d} B_s^i+\int_{0}^{t}\int_{\mathbb{R}^d\backslash\{0\}}\beta(s,  X_{s^-}^i,\mathcal{L}_{X_{s^-}^i},z)\tilde{N}^i(\mathrm{d}s, \mathrm{d}z).$$    
		Then, 		$$
		\lim _{N \rightarrow \infty} \max _{1 \leq i \leq N} \mathbb{E}\left( \sup _{0 \leq t \leq T}|X_t^{i, N}-X_t^i|^2\right)=0.
		$$
	\end{theorem}
	\begin{proof}
		Note that $C$ denotes a constant that may vary from line to line and it is independent of $N$. Define $$\mu^N= \frac{1}{N}\sum_{j=1}^N\delta_{X^{j,N}},\quad \text{and} \quad\mu= \frac{1}{N}\sum_{j=1}^N\delta_{X^{j}}.$$
		
		Let $D$ be the family of time-dependent sets in $\mathbb{R}^d$ defined by $D=\{\Phi^{-1}_t(\mathcal{D}),\;t\in[0,T]\}$. We can see that if $(X^{i,N},K^N)$ and $(X^{i},k)$ are solutions to $\textbf{\upshape (S)}$ and $RP_{\mathcal{D}}(\phi,Y^i)$, respectively. Then $\left(l^*\left( \frac{1}{N}\sum_{j=1}^NX^{j,N}\right), l^*(K^N)\right)=RP_{D}\left(l^*\left( \frac{1}{N}\sum_{j=1}^NY^{j,N}\right)\right)$  and $\left(l^*(\mathbb{E}(X^j)), l^*(k)\right)=RP_{D}\left(l^* (\mathbb{E}(Y^{j})\right)$,
		where 
		\begin{equation*}
			Y^{i,N}:=X_{0}^i+ \int_{0}^{t} b\left(s, X_s^{i,N},\mu^N_s\right) \mathrm{d}s+ \int_{0}^{t} \sigma\left(s, X_s^{i,N},\mu^N_s\right) \mathrm{d} B_s^i+\int_{0}^{t}\int_{\mathbb{R}^d\backslash\{0\}}\beta\left(s,  X_{s^-}^{i,N}, \mu^N_{s^-},z\right)  \tilde{N}^i(\mathrm{d} s, \mathrm{d} z)
		\end{equation*} 
		We set
		$$\tilde{X}^N= l^*\left( \frac{1}{N}\sum_{j=1}^NX^{j,N}\right),\quad \tilde{K}^N= l^*(K^N),\quad \hat{X}^N= l^*\left( \frac{1}{N}\sum_{j=1}^NX^{j}\right), \;  \hat{k}= l^*(k),$$
		
		$$\tilde{Y}^N= l^*\left( \frac{1}{N}\sum_{j=1}^NY^{j,N}\right), \;\text{and}\; \hat{Y}^N= l^*\left( \frac{1}{N}\sum_{j=1}^NY^{j}\right)$$
		\\
		\textbf{Step 1:} We show that there exists a constant $C$, independent of $N$, such that
		\begin{equation*}
			\mathbb{E}\left(\sup_{s\in [0,T]}|X_s^{j,N}|^4+ \sup_{s\in [0,T]}|X_s^{j}|^4+ |K^N|_T^2\right)\leq C.
		\end{equation*}
		We have \begin{align}
			\mathbb{E}\left(\sup_{s\in [0,T]}|X^{i,N}_t|^4\right)&\leq 6\mathbb{E}\left( \sup_{s\in [0,T]}|Y^{i,N}_t|^4+ \sup_{s\in [0,T]}|K^{N}_t|^4 \right)\nonumber\\
			&\leq 6\mathbb{E}\left( \sup_{s\in [0,T]}|Y^{i,N}_t|^4+\|(l^*)^{-1}\| \sup_{s\in [0,T]}|\tilde{K}^{N}_t|^4 \right)   \label{i_1}
		\end{align}
		by Proposition \ref{estim_quad_Ap} in the appendix, and using the exchangeability of $(X^{i,N})_{\{i=1,\cdots,N\}}$, along with the assumptions $\left(\mathbf{A_1^p}\right)$-$\left(\mathbf{A_3^p}\right)$, we obtain
		
		\begin{align*}
			\mathbb{E}\left( \sup_{s\in [0,T]}|\tilde{K}^{N}_t|^4\right)&\leq C \mathbb{E}\left(|X_0^i|^4+\int_{0}^{T} (1+\sup_{0 \leq s \leq u}|X_u^i|^4) \mathrm{d} s+\int_{0}^{T} |b(s,0,\mu_s)|^4\mathrm{d}s+\int_{0}^{T} |\sigma(s,0,\mu_s)|^4\mathrm{d}s\right)
		\end{align*}
		we can estimate $ \mathbb{E}\left( \sup_{s\in [0,T]}|Y^{i,N}_t|^4\right)$ by a similar bound, using Burkholder--Davis--Gundy inequality. Combining this with \eqref{i_1} and applying Gronwall's lemma, we obtain that $\mathbb{E}\left(\sup_{t\in [0,T]}|X^{i,N}_t|^4\right)$ is uniformly bounded. 
		Moreover, 
		\begin{align*}
			\mathbb{E}\left(|K^{N}|_T^2\right)\leq  \|(l^*)^{-1}\|\mathbb{E}\left(|\tilde{K}^{N}|_T^2\right)
		\end{align*}
		using the fact that $\tilde{Y}=\tilde{X}-\tilde{K}$ and the continuity of $l^*$, we apply the estimate \eqref{estim_k_} of Theorem \ref{prop_A} in the appendix, with $A_t=\Phi_{t}^{-1}(\alpha)$ for some $\alpha\in\mathring{\mathcal{D}}$, and obtain
		\begin{equation*}
			\mathbb{E}\left(|K^{N}|_T^2\right)\leq C \mathbb{E}\left(\sup_{s\in [0,T]}|X_s^{j,N}|^4+ \sup_{s\in [0,T]}|K_s^{N}|^4+ \sup_{s\in [0,T]}|v_t|^4+1\right).
		\end{equation*}
		which is uniformly bounded in $N$ by the preceding estimates.
		
		Finally, using Proposition \ref{estim_quad} and applying similar arguments as above, we also conclude that $\mathbb{E}\left(\sup_{t\in [0,T]}|X^{i}_t|^4\right)$ is finite.
		\\
		\textbf{Step 2:} Now, we show $$\lim _{N \rightarrow \infty} \max _{1 \leq i \leq N} \mathbb{E}\left( \sup _{0 \leq t \leq T}|X_t^{i, N}-X_t^i|^2\right)=0.$$
		Let $i\in \{1,\cdots,N\}$. We have
		\begin{align*}
			\mathbb{E}\left(\sup_{0 \leq t \leq T}|X^{i,N}_t-X_t^i|^2\right)\leq 2\left( \mathbb{E}(\sup_{0 \leq t \leq T}|Y^{i,N}_t-Y_t^i|^2)+ \mathbb{E}(\sup_{0 \leq t \leq T}|K^N_t-k_t|^2)\right)
		\end{align*}
		For the second term, we estimate
		\begin{align*}
			\mathbb{E}\left(\sup_{0 \leq t \leq T}|K^N_t-k_t|^2\right)&\leq \|(l^*)^{-1}\|\mathbb{E}\left(\sup_{0 \leq t \leq T}|\tilde{K}^N_t-\hat{k}_t)|^2\right)\\
			&\leq C \left(\mathbb{E}\left(\sup_{0 \leq t \leq T}|\tilde{X}^N_t-\hat{X}^N_t|^2+\sup_{0 \leq t \leq T}|\hat{X}^N_t-\mathbb{E}(\hat{X}^N_t)|^2\right) \right) \\
			&+ C\left(\mathbb{E}\left(\sup_{0 \leq t \leq T}|\tilde{Y}^N_t-\hat{Y}^N_t|^2+\sup_{0 \leq t \leq T}|\hat{Y}^N_t-\mathbb{E}(\hat{Y}^N_t)|^2\right) \right).
		\end{align*}
		In the last line, we used the fact that $(\tilde{X}^N,\tilde{K}^N)=RP_D(\tilde{Y}^N)$ and $(\mathbb{E}(\hat{X}^N_t),\hat{k})=RP_D(\mathbb{E}(\hat{Y}^N))$, and we added and subtracted $\hat{X}^N$. Since $k$ is deterministic and $(Y^{i,N},Y^i)_{\{i=1,\cdots,N\}}$ are exchangeable, it follows that
		\begin{equation}
			\mathbb{E}\left(\sup_{0 \leq t \leq T}|X^{i,N}_t-X_t^i|^2\right)\leq C\mathbb{E}\left(\sup_{0 \leq t \leq T}|\tilde{X}^N_t-\hat{X}^N_t|^2+\sup_{0 \leq t \leq T}|\hat{X}^N_t-\mathbb{E}(\hat{X}^N_t)|^2+\sup_{0 \leq t \leq T}|Y^{i,N}_t-Y_t^i|^2\right). \label{eq0}
		\end{equation}
		
		Applying It\^o's formula to $\tilde{X}^N-\hat{X}^N$, we obtain
		\begin{align*}
			|\tilde{X}^N_t-\hat{X}^N_t|^2&\leq \frac{2}{N}\sum_{j=1}^N\int_0^t \langle l(\tilde{X}_s^N-\hat{X}^N_s), b(s,X_s^{j,N},\mu_s^N)-b(s,X_s^{j},\mathcal{L}_{X_s^j})\rangle\mathrm{d}s\\
			&+ \frac{2}{N}\sum_{j=1}^N\int_0^t \langle l(\tilde{X}_s^N-\hat{X}^N_s), \sigma(s,X_s^{j,N},\mu_s^N)-\sigma(s,X_s^{j},\mathcal{L}_{X_s^j})\mathrm{d}B_s\rangle\\
			&+ \frac{2}{N}\sum_{j=1}^N\int_0^t\int_{\mathbb{R}^d\backslash{\{0\}}} \langle l(\tilde{X}_{s^-}^N-\hat{X}^N_{s^-}), \beta(s,X_{s^-}^{j,N},\mu_{s^-}^N)-\beta(s,X_{s^-}^{j},\mathcal{L}_{X^j_{s^-}})\rangle\tilde{N}^j(\mathrm{d}s,\mathrm{d}z)\\
			&+2\int_0^t \langle \tilde{X}^N_s-\hat{X}^N_s, \mathrm{d}\tilde{K}^N_s\rangle+2\int_0^t \langle\hat{X}^N_s- \tilde{X}^N_s, \mathrm{d}\hat{k}_s\rangle\\
			&+\frac{1}{N}\sum_{j=1}^N\int_0^t\int_{\mathbb{R}^d\backslash{\{0\}}} | l(\beta(s,X_{s^-}^{j,N},\mu_{s^-}^N)-\beta(s,X_{s^-}^{j},\mathcal{L}_{X_{s^-}^j}))|^2N^j(\mathrm{d}s,\mathrm{d}z)\\
			&=I_1+I_2+I_3+I_4+I_5+I_6.
		\end{align*}
		We have
		\begin{align*}
			I_4\leq 2\int_0^t \langle \tilde{X}^N_s- \mathbb{E}(\hat{X}^N_s), \mathrm{d}\tilde{K}^N_s\rangle + 2\int_0^t \langle \mathbb{E}(\hat{X}^N_s)- \hat{X}^N_s, \mathrm{d}\tilde{K}^N_s\rangle
		\end{align*}
		since $(\mathbb{E}(\hat{X}^N),\hat{k})=RP_D(\mathbb{E}(\hat{Y}^N))$, $\mathbb{E}(\hat{X}^N_s)\in D_s$. Moreover, since $(\tilde{X}^N,\tilde{K}^N)=RP_D(\tilde{Y}^N)$, it follows that
		\begin{equation*}
			I_4\leq 2\int_0^t \langle \mathbb{E}(\hat{X}^N_s)- \hat{X}^N_s, \mathrm{d}\tilde{K}^N_s\rangle \leq2\sup_{s\in[0,T]}|\mathbb{E}(\hat{X}^N_s)- \hat{X}^N_s||\tilde{K}^N|_T
		\end{equation*}
		Similarly, we obtain
		\begin{equation*}
			I_5\leq2\sup_{s\in[0,T]}|\mathbb{E}(\hat{X}^N_s)- \hat{X}^N_s||\hat{k}|_T.
		\end{equation*}
		Now, taking the supremum and the expectation, and applying Burkholder--Davis--Gundy inequality, the Lipchitz property of the coefficients, and by  applying the triangle inequality for the Wasserstein distance with respect to the measure $\mu$, and the exchangeability of $(X^{i,N},X^i)$, we obtain
		\begin{equation}
			\mathbb{E}(I_1+I_2+I_3+I_6)\leq \frac{1}{2}\mathbb{E}(\sup_{t\in[0,T]}|X_t^{i,N}-X_t^{i}|^2)+ C \mathbb{E}\left(\int_0^T  \sup_{r\in[0,s]}|X_r^{i,N}-X_r^{i}|^2+ W_2(\mu_s,\mathcal{L}_{X_s^i})\mathrm{d}s\right).\label{eq1}
		\end{equation}
		Let $(t_m)_{m=1,\dots,M}$ be a uniform partition of the interval $[0,T]$, with $t_m=\frac{mT}{M}$, for $m=1,\cdots,M$. Then
		\begin{align}
			\mathbb{E}(I_5+I_6) &\leq \mathbb{E}\left(\sup_{s\in[0,T]}|\mathbb{E}(\hat{X}^N_s)- \hat{X}^N_s|^2\right)^{\frac{1}{2}}\left(\mathbb{E}(|K^N|_T^2)^{\frac{1}{2}}+|k|_T\right)\nonumber\\
			&\leq C\max_{m=1,\cdots,M}\mathbb{E}\left(|\mathbb{E}(\hat{X}^N_{t_m})- \hat{X}^N_{t_m}|^2\right)^{\frac{1}{2}}\nonumber\\
			&+ C\max_{m=1,\cdots,M}\mathbb{E}\left(\sup_{s\in[t_m,t_{m+1}[}|\hat{X}^N_{t_m}-\hat{X}^N_{s}|^2\right)^{\frac{1}{2}}\nonumber\\
			&\leq  C\left(\max_{m=1,\cdots,M}\mathbb{E}\left(|\mathbb{E}(\hat{X}^N_{t_m})- \hat{X}^N_{t_m}|^2\right)^{\frac{1}{2}}+ c\left(\frac{1}{M}\right)\right).
			\label{eq3}
		\end{align}
		Here, the constant $C$ is depending on the constant in Step 1, and $c(\cdot)$ is the function from Proposition \ref{estim-VAR},
		Morover, the term $\mathbb{E}(\sup_{0 \leq t \leq T}|Y^{i,N}_t-Y_t^i|^2)$ can be estimated similarly to \eqref{eq1}.
		Combining the estimates \eqref{eq1} \eqref{eq3}, and \eqref{eq0}, and applying Gronwall's Lemma, we obtain
		\begin{align}
			\max _{1 \leq i \leq N}\mathbb{E}(\sup_{t\in[0,T]}|X_t^{i,N}-X_t^{i}|^2)&\leq C \max _{1 \leq i \leq N}\mathbb{E}\left(\int_0^T W_2(\mu_s,\mathcal{L}_{X_s^i})\mathrm{d}s\right)\nonumber\\
			&+ C\left( \max_{m=1,\cdots,M}\mathbb{E}\left(|\mathbb{E}(\hat{X}^N_{t_m})- \hat{X}^N_{t_m}|^2\right)^{\frac{1}{2}}+ c\left(\frac{1}{M}\right)\right). \label{eq00}
		\end{align}
		We have $X^i\in\mathcal{S}^2_{0,T}$. Then, the convergence of the first term on the right-hand side of \eqref{eq00} follows from the discussion on page $361$ of \cite{Delarue}.\\
		By the strong law of large numbers, we have $\lim\limits_{N}|\mathbb{E}(\hat{X}^N_{t_m})- \hat{X}^N_{t_m}|^2=0$ $\mathbb{P}$-a.s. Moreover, by Step 1, the family $\{|\hat{X}_{t_m}^N|^2,\;m\in\{1,\dots,M\}\}$ is uniformly integrable, which implies $$\lim\limits_{N}\mathbb{E}\left(|\mathbb{E}(\hat{X}^N_{t_m})- \hat{X}^N_{t_m}|^2\right)=0.$$
		In \eqref{eq00}, we first let $N\to\infty$, then $M\to\infty$. Using the fact that the function $c(\cdot)$ is continuous at $0$, we conclude the result.
	\end{proof}
	\section{Optimal control problem for McKean SDEs with mean reflection}\label{control}
	In this section, we study optimal control problems for McKean–Vlasov SDEs with mean reflection.
	We consider a Radon measure $\lambda$ on $\mathbb{R}^d \setminus {\{0\}}$ that satisfies the Lévy condition and whose support is contained in a compact set $\Gamma$, and let $x\in\mathcal{D}$.
	\subsection{Formulation of the problem}
	We define the notion of control as the following.
		\begin{definition}\label{def of Strict Control}
			We call an admissible system control a set $\alpha = (\Omega,\mathcal{F},\mathbb{F},\mathbb{P}, B, \tilde{N}, u, X, k)$ such that
			\begin{itemize}
				\item[i)]  $(\Omega,\mathcal{F},\mathbb{P})$ is a probability space equipped with a filtration $\mathbb{F}=(\mathcal{F}_t)_{t\geq 0}$;
				\item[ii)] $u$ is a $\mathbb{A}$-valued process, $\mathbb{F}$-progressively measurable;
				\item[iii)]  $ B$  is a standard Brownian motion and $\tilde{N}$ is a compensated Poisson random measure with Lévy measure $\lambda$, both defined on the filtered probability space $(\Omega, \mathcal{F}, \mathbb{F}, \mathbb{P})$.
				\item[iv)]$k$ is a deterministic mapping in $\mathbb{C}([0,T], \mathbb{R}^d)$, and $X$ is an $\mathbb{F}$-adapted process in $\mathbb{D}([0,T], \mathbb{R}^d)$ satisfy the following;
				\begin{equation}\label{eqdef}
					\begin{cases}
						X_t= x+ \displaystyle\int_{0}^{t} b(s, X_s,\mathcal{L}_{X_s}, u_s)\mathrm{d}s + \displaystyle\int_{0}^{t}\sigma(s,X_s, \mathcal{L}_{X_s}) \mathrm{d}B_s  \\
						\quad\quad+\displaystyle\int_{0}^{t}\int_{\mathbb{R}^d\backslash\{0\}}\beta\left(s,X_{s^-}, \mathcal{L}_{X_{s^-}},z\right)  \tilde{N}(\mathrm{d} s, \mathrm{d} z)+ k_t,\\
						\mathbb{E}^{\mathbb{P}}(X_t) \in \mathcal{D}, \\
						\text{for all}\;z\in \mathcal{D}, \displaystyle\int_0^t \langle \mathbb{E}^{\mathbb{P}}(X_{s}) - z,  \mathrm{d}k_s \rangle \leq 0 . 
					\end{cases}
				\end{equation}
			\end{itemize}
			We denote by $\mathcal{U}_{ad}$ the set of all admissible controls.
		\end{definition}
		The cost functional corresponding to a control $u \in \mathcal{U}_{ad}$ is defined as the following:
		\begin{equation*}
			J(u) = \mathbb{E}^{\mathbb{P}}\left(\int_0^T  f(s, X_s, \mathcal{L}_{X_s}, u_s)\mathrm{d}s + \int_0^T h(s,X_s,\mathcal{L}_{X_s})\mathrm{d}|k|_s + g(X_T)\right).
		\end{equation*} 
		The goal is to minimize the cost functional $J$ over the set of admissible controls \( \mathcal{U}_{ad} \); that is, to find an optimal control $ u^* \in \mathcal{U}_{ad} $ such that
		$$
		J(u^*) = \min_{u \in \mathcal{U}_{ad}} J(u).
		$$
		
		We consider the following set of assumptions $(\mathbf{A})$:
		\begin{itemize}
			\item[$\left(\mathbf{A_1}\right)$] The mappings $b:[0, T] \times \mathbb{R}^d \times \mathcal{P}\left(\mathbb{R}^d\right)\times \mathbb{A} \rightarrow \mathbb{R}^d$, $\sigma:[0, T] \times \mathbb{R}^d \times\mathcal{P}\left(\mathbb{R}^d\right)\rightarrow \mathbb{R}^m$, and $\beta:[0, T]\times\mathbb{R}^d \times\mathcal{P}\left(\mathbb{R}^d\right)\times \mathbb{R}^{d}\backslash{\{0\}} \rightarrow \mathbb{R}^d$ are continuous. Moreover, there exists $\gamma>0$ such that for all $t\in [0,T]$, $a\in\mathbb{A}$, $x, x^{\prime}\in\mathbb{R}^d$, and $\mu,\mu^{\prime}\in\mathcal{P}(\mathbb{R}^d)$, the following hold:
				\begin{multline*}
					\left|b(t,x,\mu,a)-b(t,x^{\prime},\mu^{\prime},a)\right|+\left|\sigma(t,x,\mu)-\sigma(t,x^{\prime},\mu^{\prime})\right|+ \left|\beta(t,x,\mu,z)-\beta(t,x^{\prime},\mu^{\prime},z)\right|\\\leq \gamma\left(|x-x^{\prime}|+W_2(\mu, \mu^{\prime})\right).\end{multline*}
				$$ |b(t,x,\mu,a)|+|\sigma(t,x,\mu)|+|\beta(t,x,\mu, z)|\leq \gamma(1+|x|). $$
				\item[$\left(\mathbf{A_2}\right)$] The mappings
				\begin{equation*}
					f :[0, T] \times \mathbb{R}^d \times \mathcal{P}\left(\mathbb{R}^d\right) \times \mathbb{A} \to \mathbb{R}, \quad h :[0, T]\times \mathbb{R}^d \times \mathcal{P}\left(\mathbb{R}^d\right) \to \mathbb{R}, \quad g : \mathbb{R}^d \to \mathbb{R},
				\end{equation*}
				are continuous and exhibit linear growth in $x$, uniformly with respect to $(t,\mu,a)$. Additionally, $f$ and $h$ are Lipschitz continuous, uniformly in $(t, a)$.
		\end{itemize}
		\begin{remark}
			\begin{itemize}
				\item 
				Using an argument similar to that of Theorem \ref{Main_th}, we conclude that $\mathcal{U}_{ad}$ is nonempty. Indeed,
				let $ (\Omega, \mathcal{F}, \mathbb{F}, \mathbb{P})$ be a filtered probability space. Let $B$ be a standard Brownian motion and $ \tilde{N}$ a compensated Poisson random measure, both defined on this space. Under Assumption $\left(\mathbf{A_1}\right)$, Theorem \ref{Main_th} ensures that for any fixed constant control $u_0 \in \mathcal{U}$, the following system:
				\[
				\begin{cases}
					X_t = x + \displaystyle\int_{0}^{t} b(s, X_s, \mathcal{L}_{X_s}, u_0) \, \mathrm{d}s + \displaystyle\int_{0}^{t} \sigma(s, X_s, \mathcal{L}_{X_s}) \, \mathrm{d}B_s \\
					\quad\quad + \displaystyle\int_{0}^{t} \int_{\Gamma} \beta(s, X_{s^-}, \mathcal{L}_{X_{s^-}}, z) \, \tilde{N}(\mathrm{d}s, \mathrm{d}z) + k_t, \\
					\mathbb{E}^{\mathbb{P}}(X_t) \in \mathcal{D}, \\
					\text{for all } z \in \mathcal{D}, \quad \displaystyle\int_0^t \langle \mathbb{E}^{\mathbb{P}}(X_s) - z, \, \mathrm{d}k_s \rangle \leq 0,
				\end{cases}
				\]
				admits a unique strong solution \( (X^{u_0}, k^{u_0}) \).
				
				Now, define a constant control process \( u_t := u_0 \) for all \( t \in [0, T] \). Then it follows that the tuple \( (\Omega, \mathcal{F}, \mathbb{F}, \mathbb{P}, B, \tilde{N}, u, X^{u_0}, k^{u_0}) \) belongs to the set of admissible controls $ \mathcal{U}_{ad}$.
				
				\item Note that under assumption $\left(\mathbf{A_2}\right)$, the mapping $J$ is well-define.
			\end{itemize}
			
		\end{remark}
		
		We now extend Definition~\ref{def of Strict Control} to the framework of relaxed controls.
		\begin{definition}\label{def of relaxed control}
			A tuple \(\gamma = (\Omega, \mathcal{F}, \mathbb{F}, \mathbb{P}, B, \tilde{N}, q, X, k)\) is called a \emph{relaxed control} if it satisfies Conditions $i)$, and $iii)$ of Definition~\ref{def of Strict Control}, and the following:
			
			\begin{itemize}
				\item[iii)] $q$ is $\mathcal{V}$-valued random variable such that for each $t\in[0,T]$, $\mathds{1}_{]0,t]}q$ is $\mathcal{F}_t$- measurable;
				\item[iv)] $k$ is a deterministic mapping in $\mathbb{C}([0,T], \mathbb{R}^d)$, and $X$ is an $\mathbb{F}$-adapted process in $\mathbb{D}([0,T], \mathbb{R}^d)$ satisfy the following:
				$$\left\{\begin{aligned}
					&X_t= x+ \int_{0}^{t}\int_{\mathbb{A}} b(s, X_s,\mathcal{L}_{X_s},a)q_s(\mathrm{d}a)\mathrm{d}s + \int_{0}^{t}\sigma(s,X_s, \mathcal{L}_{X_s}) \mathrm{d}B_s  \\
					& \quad\quad+\int_{0}^{t}\int_{\mathbb{R}^d\backslash\{0\}}\beta\left(s,  X_{s^-}, \mathcal{L}_{X_{s^-}},z\right)  \tilde{N}(\mathrm{d} s, \mathrm{d} z)+ k_t,\\
					&
					\mathbb{E}^{\mathbb{P}}(X_t) \in \mathcal{D}, \\
					& \text{for all}\;z\in \mathcal{D}, \int_0^t \langle \mathbb{E}^{\mathbb{P}}(X_{s}) - z,  \mathrm{d}k_s \rangle \leq 0 .
				\end{aligned}
				\right.
				$$
			\end{itemize}
			We denote by \(\mathcal{R}\) the set of all relaxed controls.
		\end{definition}
		\begin{remark}
			Under Assumption $\left(\mathbf{A_1}\right)$, the set \(\mathcal{R}\) of relaxed controls is nonempty. Indeed, by taking \(q_t := \delta_{a_0}\) for some fixed \(a_0 \in \mathbb{A}\), and applying Theorem~\ref{Main_th}, the existence of such a relaxed control follows directly.
		\end{remark}
		
		The cost functional corresponding to a relaxed control $\gamma \in \mathcal{R}$ is defined as follows:
		\begin{equation*}
			J(\gamma) := \mathbb{E}^{\mathbb{P}}\left(\int_0^T \int_{\mathbb{A}} f(s, X_s, \mathcal{L}_{X_s}, a) q_s(\mathrm{d}a)\mathrm{d}s + \int_0^T h(s,X_s,\mathcal{L}_{X_s})\mathrm{d}|k|_s + g(X_T)\right).
		\end{equation*} 
		\subsection{Existence of optimal controls}\label{sec2}
		\begin{theorem} \label{theooptimal}
			Under assumptions $\left(\mathbf{A}\right)$, the relaxed control problem admits an optimal solution. Furthermore, if the following condition is also satisfied:
			\begin{description}
				\item[ (\textit{Roxin’s Condition}):] For every $ (t,x,\mu) \in [0, T] \times \mathbb{R}^d\times\mathcal{P} $, the set 
				\begin{equation*}
					(b, f)(t, x, \mu,\mathbb{A}) := \left\{\big(b(t, x,\mu, a), f(t, x,\mu, a)\big) : a \in \mathbb{A} \right\}
				\end{equation*}
				is convex and closed in $ \mathbb{R}^{d+1}$.
			\end{description}
			Then the relaxed control problem admits a strict optimal control.
		\end{theorem}
		Let $(\gamma_n)_{n \geq 0}$ be a minimizing sequence, such that $\lim_{n \to \infty} J(\gamma_n) = \inf_{\mu \in \mathcal{R}} J(\mu)$. Let $(X^n, k^n)$ be the unique solution of the following system:
		$$\left\{\begin{aligned}
			&X_t^n= x+ \int_{0}^{t}\int_{\mathbb{A}} b(s, X_s^n,\mathcal{L}_{X_s^n},a)q_s^n(\mathrm{d}a)\mathrm{d}s + \int_{0}^{t}\sigma(s,X_s^n, \mathcal{L}_{X_s^n}) \mathrm{d}B_s^n  \\
			& \quad\quad+\int_{0}^{t}\int_{\mathbb{R}^d\backslash\{0\}}\beta\left(s,  X_{s^-}^n, \mathcal{L}_{X_{s^-}^n},z\right)  \tilde{N}^n(\mathrm{d} s, \mathrm{d} z)+ k_t^n, \\
			&
			\mathbb{E}(X_t^n) \in \mathcal{D}, \\
			& \text{for all}\;z\in \mathcal{D},\; \;\int_0^t \langle \mathbb{E}(X_{s}^n) - z,  \mathrm{d}k^n_s \rangle \leq 0 .
		\end{aligned}
		\right.
		$$
		In what follows, $C$ is a constant that may vary from line to line. 
		\begin{lemma} \label{optimallemma1}
			There exists a positive constant C such that 
			\begin{itemize}
				\item [i)] $\displaystyle\sup_n \left(\mathbb{E}(\sup_{0 \leq 0 \leq T} |X^n_t |^2)+\sup_{0 \leq 0 \leq T} |k^n_t |^2+|k^n|_{T}\right)  \leq C.$
				\item[ii)] There exists a constant $C$ such that for all $s,t\in[0,T]$, 
				$$\sup_n\left(|k^n_t-k_s^n |^2+\mathbb{E}(|X^n_t-X_s^n |^2)\right)\leq C|t-s|.
				$$
			\end{itemize}
		\end{lemma}
		\begin{proof}
			\begin{itemize}
				\item[$i)$] 
				From the inequality \eqref{ineq_estim_quad} in Proposition \ref{estim_quad}, we obtain
				\begin{align*}
					\mathbb{E}\left(\sup_{0 \leq t \leq T} |X^n_t-x|^2\right)&\leq C \mathbb{E}\left( \left(\int_0^T\int_{\mathbb{A}}|b(s,X_s^n,\mathcal{L}_{X_s^n},a)|q_s^n(\mathrm{d}a)\mathrm{d}s\right)^2+\int_0^T |\sigma(s,X_s^n,\mathcal{L}_{X_s^n})|^2\mathrm{d}s\right)\\
					&+C\mathbb{E}\left(\int_0^T \int_{\mathbb{R}^d\setminus {\{0 \}}}|\beta(s,X_{s^-}^n,\mathcal{L}_{X_{s^-}^n},z)|^2\lambda(\mathrm{d}z)\mathrm{d}s\right).
				\end{align*}
				We apply Jensen's inequality to the first term, and we use the fact that $b$, $\sigma$, and $\beta$ has linear growth, we obtain
				\begin{equation*}
					\mathbb{E}\left(\sup_{0 \leq t \leq T} |X^n_t|^2\right)\leq C\left(1+|x|^2+\int_0^T \sup_{0 \leq r \leq s} |X^n_r|^2 \mathrm{d}s\right),
				\end{equation*}
				we apply Gronwall's Lemma and deduce that $\mathbb{E}\left(\sup\limits_{0 \leq t \leq T} |X^n_t|^2\right)$ is uniformly bounded in $n$.
				
				By similar arguments, the result also holds for $\sup\limits_{0 \leq t \leq T} |k^n_t|^2$, and for $|k^n|_T$ by applying inequality \eqref{estim_var_k} to the latter.
				\item[$ii)$]Let $0\leq s\leq t\leq T$. Arguing similarly to the proof of Proposition \ref{estim-VAR}, but this time using the fact that $b$,$\sigma$, and $\beta$ have linear growth, we obtain:
				\begin{equation*}
					\mathbb{E}\left(|X_t^n-X_s^n|^2\right)\leq C(t-s)\left(1+\mathbb{E}\left(\sup_{0\leq s\leq T }|X_s^n|^2\right)\right)
				\end{equation*}
				from statement $(i)$, $\mathbb{E}\left(\sup\limits_{0 \leq t \leq T} |X^n_t|^2\right)$ is uniformly bounded in $n$, hence the inequality holds.
			\end{itemize}
		\end{proof}
		\begin{lemma}\label{optimallemma3}
			The sequence of processes $\gamma^n:= (X^n, B^n,\tilde{N}^n, q^n)$ is tight in the space $$S=\mathbb{D}([0,T], \mathbb{R}^d) \times \mathbb{C}([0,T], \mathbb{R}^d) \times \mathcal{V} \times \mathbb{D}([0,T],\mathcal{M}(\Gamma)),$$ endowed with the product topology, where the first factor is equipped with the $d_0$ metric, the second with the topology of uniform convergence, the third with the stable topology on measures, and the last with the $d_0$ metric.\\ Here, $\mathcal{M}(\Gamma)$ denotes the space of finite measures on $\Gamma$.
			
		\end{lemma}
		\begin{proof}
			The tightness of $(X^n)$ follows from Aldous's criterion, together with Lemma \ref{optimallemma1}.\\ Since $\Gamma$ is compact, the tightness of $\tilde{N}^n$ follows from Theorem 6.2 in \cite{kushner2}.\\ Moreover  $\mathbb{A}$ is compact, the set $\mathcal{V}$ is a closed subset of the space of finite measures on $[0,T]\times\mathbb{A}$ having uniformly bounded total mass. It follows that $\mathcal{V}$ is compact for the stable topology, which in this setting coincides with the vague and weak topologies. By Prokhorov’s theorem, the space $\mathcal{P}(\mathcal{V})$ is also compact. In particular, $(q^n)$ is tight.
		\end{proof}
		
		\begin{proof}[Proof of Theorem \ref{theooptimal}]
			From Lemmas \ref{optimallemma3}, 
			there exists a sub-sequence $(\gamma^{n_k})$ of $(\gamma^n)$, still denoted by $(\gamma^n)$, that converges weakly. By the Skorokhod representation theorem, there exists a probability space $(\hat{\Omega},\hat{\mathcal{F}}, \hat{\mathbb{P}})$, a sequence $\hat{\gamma}^n = (\hat{X}^n, \hat{B}^n, \hat{q}^n, \widehat{\tilde{N}^n})$ and $\hat{\gamma}=(\hat{X},\hat{B}, \hat{q},\widehat{\tilde{N}})$ defined on this space such that:
			\begin{itemize}
				\item[(i)] For each $n \in \mathbb{N}$, we have $\mathcal{L}_{\hat{\gamma}^n}= \mathcal{L}_{\gamma^n}$.
				\item[(ii)] The suequence $(\hat{\gamma}^n)$ converges to $(\hat{\gamma})$ $\hat{\mathbb{P}}$-a.s. in the space $S$.
			\end{itemize}
			From Lemma \ref{optimallemma1}, we conclude, by Azelà-Ascoli theorem, that there exists a subsequence of $(k^n)$, still denoted by $(k^n)$, that converges uniformly to some function $k$, and such that $|k^n|$ converge uniformly to $|k|$. 
			
			From (i), we have 
			$$\left\{\begin{aligned}\label{optimaleq2}
				&\hat{X}_t^n= x+ \int_{0}^{t}\int_{\mathbb{A}}b(s, \hat{X}^n_s,\mathcal{L}_{\hat{X}^n_s},a)\hat{q}_s^n(\mathrm{d}a)\mathrm{d}s+ \int_{0}^{t}\sigma(s,\hat{X}^n_s, \mathcal{L}_{\hat{X}^n_s})\mathrm{d}\hat{B}_s^n \\
				& \quad\quad+\int_{0}^{t}\int_{\mathbb{R}^d\backslash\{0\}}\beta\left(s,\hat{X}^n_{s^-}, \mathcal{L}_{\hat{X}^n_{s^-}},z\right)  \widehat{\tilde{N}^n}(\mathrm{d} s, \mathrm{d} z)+ k_t^n, \\
				&
				\hat{\mathbb{E}}(\hat{X}^n_t) \in \mathcal{D},\\
				& \text{for all}\;z\in \mathcal{D},\; \int_0^t \langle \hat{\mathbb{E}}(\hat{X}^n_{s})- z, \mathrm{d}k^n_s \rangle \leq 0 .
			\end{aligned}
			\right.
			$$
			To establish the limit, we will prove the following convergences hold in probability 
			\begin{equation}
				\int_{0}^{T}\int_{\mathbb{A}}b(s,\hat{X}^n_s,\mathcal{L}_{\hat{X}^n_s},a)\hat{q}_s^n(\mathrm{d}a)\mathrm{d}s \to  \int_{0}^{T}\int_{\mathbb{A}}b(s,\hat{X}_s,\mathcal{L}_{\hat{X}_s},a)\hat{q}_s(\mathrm{d}a)\mathrm{d}s \label{limits_1}
			\end{equation}
			
			\begin{equation}
				\int_{0}^{T}\int_{\mathbb{A}}f(s,\hat{X}^n_s,\mathcal{L}_{\hat{X}^n_s},a)\hat{q}_s^n(\mathrm{d}a)\mathrm{d}s \to  \int_{0}^{T}\int_{\mathbb{A}}f(s,\hat{X}_s,\mathcal{L}_{\hat{X}_s},a)\hat{q}_s(\mathrm{d}a)\mathrm{d}s \label{limits_2}
			\end{equation}
			
			\begin{equation*}
				\int_{0}^{T}\sigma(s,\hat{X}^n_s,\mathcal{L}_{\hat{X}^n_s})\mathrm{d}\hat{B}_s^n \to \int_{0}^{T}\sigma(s,\hat{X}_s, \mathcal{L}_{\hat{X}_s})\mathrm{d}\hat{B}_s \label{limits_3}
			\end{equation*}
			\begin{equation*}
				\int_{0}^{T}\int_{\mathbb{R}^d\backslash\{0\}}\beta\left(s,  \hat{X}^n_{s^-}, \mathcal{L}_{\hat{X}^n_{s^-}},z\right)  \widehat{\tilde{N}^n}(\mathrm{d} s, \mathrm{d} z) \to \int_{0}^{T}\int_{\mathbb{R}^d\backslash\{0\}}\beta\left(s, \hat{X}_{s^-}, \mathcal{L}_{\hat{X}_{s^-}},z\right)  \widehat{\tilde{N}}(\mathrm{d} s, \mathrm{d} z).\label{limits_4}
			\end{equation*}
			We have
			\begin{equation}
				\begin{split}
					\int_{0}^{T}\int_{\mathbb{A}}b(s,\hat{X}^n_s,\mathcal{L}_{\hat{X}^n_s},a)\hat{q}_s^n(\mathrm{d}a)\mathrm{d}s&-  \int_{0}^{T}\int_{\mathbb{A}}b(s,\hat{X}_s,\mathcal{L}_{\hat{X}_s},a)\hat{q}_s(\mathrm{d}a)\mathrm{d}s\\
					&=\int_{0}^{T}\int_{\mathbb{A}}b(s,\hat{X}^n_s,\mathcal{L}_{\hat{X}^n_s},a)-b(s,\hat{X}_s,\mathcal{L}_{\hat{X}_s},a)\hat{q}_s^n(\mathrm{d}a)\mathrm{d}s \\
					&+\int_{0}^{T}\int_{\mathbb{A}}b(s,\hat{X}_s,\mathcal{L}_{\hat{X}_s},a)(\hat{q}_s^n-\hat{q}_s)(\mathrm{d}a)\mathrm{d}s\label{twotherm}.
				\end{split}
			\end{equation}
			On the one hand, using the fact that $(\hat{q}^n)$ is a sequence of probability measures, and from the Lipschitz assumptions on $b$, we get 
			\begin{equation}\label{thefirst}
				\hat{\mathbb{E}}\left(\int_{0}^{
					T}\int_{\mathbb{A}}|b(s,\hat{X}^n_s,\mathcal{L}_{\hat{X}^n_s},a)-b(s,\hat{X}_s,\mathcal{L}_{\hat{X}_s},a)|^2\hat{q}_s^n(\mathrm{d}a)\mathrm{d}s\right)\leq C \int_0^T\hat{\mathbb{E}}(|\hat{X}_s^n-\hat{X}_s|^2)\mathrm{d}s
			\end{equation}
			moreover, we have
			$$|\hat{X}^n_t-\hat{X}_t|^2\leq 2\left(d_{\circ}(\hat{X}^n,0)^2+ d_{\circ}(\hat{X},0)^2\right) $$
			since $\lim\limits_{n} d_{\circ}(\hat{X}^n,0)^2+ d_{\circ}(\hat{X},0)^2= 2d_{\circ}(\hat{X},0)^2$, there exists $m\in\mathbb{N}$ such that for all $n\geq m$,
			$$|\hat{X}^n_t-\hat{X}_t|^2\leq 4d_{\circ}(\hat{X},0)^2+1. $$
			Using Fatou's Lemma, we obtain
			\begin{equation*}
				\hat{\mathbb{E}}\left(d_{\circ}(\hat{X},0)^2\right)\leq \lim\limits_{n} \hat{\mathbb{E}}\left(d_{\circ}(\hat{X}^n,0)^2\right)\leq  \lim\limits_{n} \sup_{n}\hat{\mathbb{E}}\left(\sup_{0\leq s\leq T}|\hat{X}^n|^2\right)\leq C
			\end{equation*}
			where $C$ is the constant from Lemma \ref{optimallemma1}.
			
			Since convergence in the $d_{\circ}$ metric implies pointwise convergence for almost every $t\in [0,T]$,
			we can apply the dominated convergence theorem and conclude that  $\lim\limits_{n} \hat{\mathbb{E}}(|\hat{X}_{t}^n-\hat{X}_{t}|^2)=0$ for a.e $t\in [0,T]$.
			
			We repeat the same arguments, we obtain that $\hat{\mathbb{E}}(|\hat{X}_{t}^n-\hat{X}_{t}|^2)$ is uniformly bounded. Hence, by dominated convergence again,
			$$\lim\limits_{n} \int_0^T\hat{\mathbb{E}}(|\hat{X}_{t}^n-\hat{X}_{t}|^2)\mathrm{d}s=0.$$
			Together with \eqref{thefirst}, this implies that the first term on the right hand side of \eqref{twotherm} converges to zero in probability.
			
			It remains to prove the convergence of the second term in \eqref{twotherm}. Let $M>0$, we have
			\begin{equation}\label{twotherm1}
				\begin{split}
					\left|\int_{0}^{t}\int_{\mathbb{A}}b(s,\hat{X}_s,\mathcal{L}_{\hat{X}_s},a)(\hat{q}_s^n-\hat{q}_s)(\mathrm{d}a)\mathrm{d}s\right|&\leq \left|\int_{0}^{t}\int_{\mathbb{A}}b(s,\hat{X}_s,\mathcal{L}_{\hat{X}_s},a)\mathds{1}_{\{|\hat{X}_{s}|< M\}}(\hat{q}_s^n-\hat{q}_s)(\mathrm{d}a)\mathrm{d}s\right|\\
					&+\left|\int_{0}^{t}\int_{\mathbb{A}}b(s,\hat{X}_s,\mathcal{L}_{\hat{X}_s},a)\mathds{1}_{\{|\hat{X}_{s}| \geq M\}}(\hat{q}_s^n-\hat{q}_s)(\mathrm{d}a)\mathrm{d}s\right|
				\end{split}
			\end{equation}
			we have that the mapping $(s,a) \mapsto b(s,\hat{X}_s,\mathcal{L}_{\hat{X}_s},a)\mathds{1}_{\{|\hat{X}_{s}|<M\}}$ is bounded and measurable, and the function $a \mapsto b(s,\hat{X}_s,\mathcal{L}_{\hat{X}_s},a)\mathds{1}_{\{|\hat{X}_{s}|<M\}}$ is continuous. From (ii), we conclude that the first term on the right-hand side of \eqref{twotherm1} converges to zero. Then, by the Lebesgue dominated convergence theorem, we obtain convergence in $L_1(\hat{\mathbb{P}})$. 
			
			On the other hand, from the linear growth assumption on $b$, we obtain 
			\begin{align*}
				&\hat{\mathbb{E}}\left(\left|\int_{0}^{T}\int_{\mathbb{A}}b(s,\hat{X}_s,\mathcal{L}_{\hat{X}_s},a)\mathds{1}_{\{|\hat{X}_{s}|\geq M\}}(\hat{q}_s^n-\hat{q}_s)(\mathrm{d}a)\mathrm{d}s\right| \right) \\
				&\leq\gamma\hat{\mathbb{E}}\left(\int_{0}^{T} \left( 1+|\hat{X}_{s}|\right) \mathds{1}_{\{|\hat{X}_{s}|\geq M\}} \mathrm{d}s \right)\\
				&\leq \frac{1}{M} \int_{0}^{T} \hat{\mathbb{E}}\left(|\hat{X}_{s}|^2+ 1 \right) \mathrm{d}s 
			\end{align*}
			where, in the last inequality we used the formula $x\mathds{1}_{\{ x>a\}}\leq \frac{x^2}{a}$, valid for $x,a>0$. Letting $M\to\infty$, and then $n\to\infty$, we obtain that the second term in \eqref{twotherm1} converges in $L_1(\hat{\mathbb{P}})$ to zero. Thus, we conclude the first convergence in \eqref{limits_1}. The remaining convergences can be shown using similar arguments and classical techniques that can be found in \cite{Situ}. Since $\hat{X}$ and $k$ are càdlàg, it follows that $\hat{\mathbb{P}}$-a.s for for every $t\in [0,T]$,
			\begin{multline*}
				\hat{X}_t =x+ \int_{0}^{t}\int_{\mathbb{A}} b(s, \hat{X}_s,\mathcal{L}_{\hat{X}_s}, a) \hat{q}_s(\mathrm{d}a)\mathrm{d}s\\ +  \int_{0}^{t}\sigma(s,\hat{X}_s,\mathcal{L}_{\hat{X}_s}) \mathrm{d}\hat{B}_s + \int_{0}^{t}\int_{\mathbb{R}^d\backslash\{0\}}\beta\left(s,  \hat{X}_{s^-}, \mathcal{L}_{\hat{X}_{s^-}},z\right)  \widehat{\tilde{N}}(\mathrm{d} s, \mathrm{d} z)+k_t.
			\end{multline*}
			For $z\in\mathcal{D}$, we show the following convergence
			\begin{equation}
				\int_0^t \langle\hat{\mathbb{E}}(\hat{X}^n_s)-z, \mathrm{d}k^n_s \rangle \to \int_0^t \langle\hat{\mathbb{E}}(\hat{X}_s)-z, \mathrm{d}k_s \rangle \label{cv_k}
			\end{equation} 
			we observe that
			\begin{multline}
				\sup_{s\in [0,T]}|\hat{\mathbb{E}}(\hat{X}^n_s-\hat{X}_s)|\leq\\\sup_{s\in [0,T]}| k^n_s -k_s|+ \hat{\mathbb{E}} \left| \int_{0}^{T}\int_{\mathbb{A}}b(s,\hat{X}^n_s,\mathcal{L}_{\hat{X}^n_s},a)\hat{q}_s^n(\mathrm{d}a)\mathrm{d}s-  \int_{0}^{T}\int_{\mathbb{A}}b(s,\hat{X}_s,\mathcal{L}_{\hat{X}_s},a)\hat{q}_s(\mathrm{d}a)\mathrm{d}s \right|\label{limit}
			\end{multline}
			We have $k^n$ converges uniformly to $k$, together with \eqref{limits_1}, this implies that $\hat{\mathbb{E}}(\hat{X}^n)$ converges uniformly to $\hat{\mathbb{E}}(\hat{X})$. By applying Lemma 5.8 in \cite{petit}, the convergence in \eqref{cv_k} follows.

			From all the results above, and since $\mathcal{D}$ is closed, we have
			$$\left\{\begin{aligned}\label{optimaleq2}
				&\hat{X}_t= x+ \int_{0}^{t}\int_{\mathbb{A}}b(s, \hat{X}_s,\mathcal{L}_{\hat{X}_s},a)\hat{q}_s(\mathrm{d}a)\mathrm{d}s+ \int_{0}^{t}\sigma(s,\hat{X}_s, \mathcal{L}_{\hat{X}_s})\mathrm{d}\hat{B}_s \\
				& \quad\quad+\int_{0}^{t}\int_{\mathbb{R}^d\backslash\{0\}}\beta\left(s,\hat{X}_{s^-}, \mathcal{L}_{\hat{X}_{s^-}},z\right)  \widehat{\tilde{N}}(\mathrm{d} s, \mathrm{d} z)+ k_t, \\
				&
				\hat{\mathbb{E}}(\hat{X}_t) \in \mathcal{D},\\
				& \text{for all}\;z\in \mathcal{D},\; \int_0^t \langle \hat{\mathbb{E}}(\hat{X}_{s})- z, \mathrm{d} k_s \rangle \leq 0 .
			\end{aligned}
			\right.
			$$
			Now, we check that $\hat{q}$ is an optimal control.\\
			Using the properties (i)-(ii) outlined above, along with assumption on cost function coefficients, we conclude that
			\begin{align*}
				\inf_{\gamma \in \mathcal{R}} J(\gamma) &= \lim_{n \to \infty} J(\gamma^n) \\
				&= \lim_{n \to \infty} \mathbb{E}\left( \int_0^T \int_{\mathbb{A}}f(s, X^n_s,\mathcal{L}_{X_s^n}, a)q^n_s(\mathrm{d} a) \mathrm{d}s + \int_0^T h(s,X^n_s,\mathcal{L}_{X_s^n})\mathrm{d}|k^n|_s + g(X^n_T) \right)\\
				&= \lim_{n \to \infty} \hat{\mathbb{E}}\left( \int_0^T \int_{\mathbb{A}} f(s, \hat{X}_s^n,\mathcal{L}_{\hat{X}_s^n}, a)\hat{q}^n_s(\mathrm{d}a) \mathrm{d}s + \int_0^T h(s,\hat{X}_s^n,\mathcal{L}_{\hat{X}_s^n})\mathrm{d}|k^n|_s + g(\hat{X}^n_T) \right) \\
				&= \hat{\mathbb{E}}\left( \int_0^T \int_{\mathbb{A}} f(s, \hat{X}_s,\mathcal{L}_{\hat{X}_s}, a)\hat{q}_s(\mathrm{d}a) \mathrm{d}s + \int_0^T h(s,\hat{X}_s, \mathcal{L}_{\hat{X}_s})\mathrm{d}|k|_s + g(\hat{X}_T) \right).
			\end{align*}
			From \eqref{limits_2}, the convergence of the first term follows. Then, using the Lipchitz assumption on $h$ and by arguments similar to those in \eqref{limit}, we obtain the uniform convergence of $\mathbb{\hat{E}}(h(\cdot,\hat{X}^n_{\cdot}, \mathcal{L}_{\hat{X}^n_{\cdot}}))$ to $\mathbb{\hat{E}}(h(\cdot,\hat{X}_{\cdot}, \mathcal{L}_{\hat{X}_{\cdot}}))$. We then conclude the convergence of the second term. Finally, by repeating the same arguments as above and applying the Lebesgue convergence theorem, we obtain the convergence of the last term.
			
			The proof of the first part is now complete.
			
			For the second part, we define the following:
			\begin{equation*}
				\hat{b}(t,\omega):=\int_\mathbb{A} b(t, \hat{X}_t(\omega),\mathcal{L}_{\hat{X}_t}, a)\hat{q}_t (\omega,\mathrm{d}a),\;\text{and}\; \hat{f}(t,\omega)=:\int_{\mathbb{A}} f(t, \hat{X}_t(\omega), \mathcal{L}_{\hat{X}_t}, a)\hat{q}_t (\omega,\mathrm{d}a).
			\end{equation*}
			$$E\left((t,\omega)\right):=\left\{ (b(t,\hat{X}_t(\omega),\mathcal{L}_{\hat{X}_s},a),f(t,\hat{X}_t(\omega),\mathcal{L}_{\hat{X}_s},a),\;\;a\in\mathbb{A} \right\}.$$
			By Roxin's conditions, we have $(\hat{b}(t,\omega),\hat{f}(t,\omega))\in E\left((t,\omega)\right)$. Therefore, by applying a measurable selection theorem (Theorem A.9 in \cite{haussmann}), there exists an a $\mathbb{A}$-valued, $\mathcal{F}^{\hat{X}, \hat{q}}$-adapted process $\hat{u}$ such that $\hat{\mathbb{P}}$-a.s. for every $t\in [0,T]$,
			\begin{equation*}
				\int_{\mathbb{A}} b(t, \hat{X}_t,\mathcal{L}_{\hat{X}_t}, a)\hat{q}_t (\mathrm{d}a) = b(t, \hat{X}_t, \mathcal{L}_{\hat{X}_t}, \hat{u}_t),\;\text{and}\; \int_{\mathbb{A}} f(t, \hat{X}_t,\mathcal{L}_{\hat{X}_t}, a)\hat{q}_t (\mathrm{d}a) = f(t, \hat{X}_t, \mathcal{L}_{\hat{X}_s}, \hat{u}_t),
			\end{equation*}
			which completes the proof.
		\end{proof}
		\subsection{Approximation of the relaxed control problem}
		In this section, we show that any relaxed control \(\gamma \in \mathcal{R}\) can be approximated by a sequence of strict controls in \(\mathcal{U}_{ad}\). Moreover, the infimum of the cost functional over strict controls coincides with the infimum over relaxed controls. This result is based on the \emph{chattering lemma}.
		
		We begin by stating the chattering lemma as established in \cite{meleard}, page 196.
		\begin{lemma}\label{chattering lemma}
			Let $\gamma = (\Omega, \mathcal{F}, \mathbb{F}, \mathbb{P}, B, \tilde{N}, q, X, k)$ be a relaxed control. Then, there exists a sequence of $\mathbb{F}$-adapted processes $(u^n_t)$ taking values in $\mathbb{A}$, such that the sequence of random measures $\delta_{u^n_t}(du)\,dt$ converges to $q_t(du)\,dt$ in $\mathcal{V}$, $\mathbb{P}$-almost surely.
		\end{lemma}
		Note that each $\gamma^n := (\Omega, \mathcal{F}, \mathbb{F}, \mathbb{P}, B, \tilde{N}, u^n, X^n, k^n)$ belongs to the set of strict controls \(\mathcal{U}_{ad}\), where \((X^n, k^n)\) is the solution of the equation \eqref{eqdef} corresponding to the control \(u^n\).
		\begin{theorem}
			Under Assumptions $\left(\mathbf{A_1}\right)$ and $\left(\mathbf{A_2}\right)$, the following hold:
			$$\displaystyle \lim_{n \to \infty} \mathbb{E}^{\mathbb{P}} \left( \sup_{0 \leq t \leq T} |X^n_t - X_t|^2 \right) + \sup_{0 \leq t \leq T} |k^n_t - k_t|^2 = 0$$
			and
			$$\displaystyle \lim_{n \to \infty} J(\gamma^n) = J(\gamma).$$
			Moreover, $$\inf_{u\in\mathcal{U}_{ad}}J(u)= \inf_{\gamma\in\mathcal{R}}J(\gamma)$$
		\end{theorem}
		\begin{proof}
			
			\begin{itemize}
				\item[i)] 
				We apply inequality \eqref{ineq_estim_quad} in Proposition \ref{estim_quad}, we get 
				\begin{align*}
					\mathbb{E}^{\mathbb{P}} \left( \sup_{0 \leq t \leq T} |X^n_t - X_t|^2 \right) + &\sup_{0 \leq t \leq T} |k^n_t - k_t|^2 \leq C\mathbb{E}^{\mathbb{P}} \left(\int_{0}^{T} \left|\sigma(s, X^n_s, \mathcal{L}_{X^n_s})-\sigma(s, X_s, \mathcal{L}_{X_s})\right|^2 \mathrm{d}s\right)\\
					&+ C\mathbb{E}^{\mathbb{P}} \left(\int_{0}^{T}\int_{\Gamma}^{} \left|\beta(s, X^n_s, \mathcal{L}_{X^n_s}, z)-\beta(s, X_s, \mathcal{L}_{X_s}, z)\right|^2 \lambda(\mathrm{d}z)\mathrm{d}s\right)\\ 
					&+ \mathbb{E}^{\mathbb{P}} \left(\int_{0}^{T} \left| \int_{\mathbb{A}}b(s, X^n_s, \mathcal{L}_{X^n_s}, u^n_s)\delta_{u_s^n}(\mathrm{d}a)-b(s, X_s, \mathcal{L}_{X_s}, a)q_s(\mathrm{d}a)\right|^2 \mathrm{d}s\right)
				\end{align*}
				By adding and subtracting $\displaystyle\int_{\mathbb{A}}b(t, X_t, \mathcal{L}_{X_t}, a)\delta_{u^n_t}(\mathrm{d}a)$ in the drift term, and using the Lipschitz assumptions on the coefficients, we fix $M>0$, we obtain
				\begin{align*}
					\mathbb{E}^{\mathbb{P}} &\left( \sup_{0 \leq t \leq T} |X^n_t - X_t|^2 \right) + \sup_{0 \leq t \leq T} |k^n_t - k_t|^2 \leq C \mathbb{E}^{\mathbb{P}} \left(\int_{0}^{T} \left| X^n_s - X_s \right |^{2} \mathrm{d}s \right)\\
					& + \mathbb{E}^{\mathbb{P}} \left(\int_{0}^{T} \left| \int_{\mathbb{A}}b(s, X_s, \mathcal{L}_{X_s}, a)\delta_{u^n_s}(\mathrm{d}a)-\int_{\mathbb{A}}b(s, X_s, \mathcal{L}_{X_s}, a) q_s(\mathrm{d}a)\right|^2 \mathds{1}_{{\{|X_{s}|\leq M\}}} \mathrm{d}s \right)\\
					&+\mathbb{E}^{\mathbb{P}} \left(\int_{0}^{T} \left| \int_{\mathbb{A}}b(s, X_s, \mathcal{L}_{X_s}, a)\delta_{u^n_s}(\mathrm{d}a)-\int_{\mathbb{A}}b(s, X_s, \mathcal{L}_{X_s}, a) q_s(\mathrm{d}a)\right|^2 \mathds{1}_{{\{|X_{s}|> M\}}} \mathrm{d}s \right)\\
					&= I_1^n+I_2^n+I_3^n.
				\end{align*}
				Since the mapping $(s,a) \mapsto b(s, X_s, \mathcal{L}_{X_s}, a) \mathds{1}_{\{|X_s| \leq M\}}$ is bounded and measurable, and continuous in $a$, it follows from Lemma~\ref{chattering lemma} that
				$$
				\mathds{1}_{\{|X_s| \leq M\}} \left| \int_{\mathbb{A}} b(s, X_s, \mathcal{L}_{X_s}, a) \, q_s(\mathrm{d}a) - \int_{\mathbb{A}} b(s, X_s, \mathcal{L}_{X_s}, a) \, \delta_{u_s^n}(\mathrm{d}a) \right| \to 0.
				$$
				By the linear growth assumption on $b$, the above expression is uniformly bounded. Applying the Dominated Convergence Theorem with respect to the product measure $\mathrm{d}\mathbb{P} \times \mathrm{d}t$, we conclude that $I_1^n \to 0$.
				
				On the other hand, since $q_s$ and $\delta_{u_s^n}$ are probability measures, and using the linear growth of $b$, we obtain
				
				$$
				I^n_2 \leq 2 \mathbb{E}^{\mathbb{P}} \left( \int_0^T \mathds{1}_{\{|X_t| > M\}} (1 + |X_t|^2) \, \mathrm{d}t \right).
				$$
				From the previous subsection, $\{{X_t,t\in [0,T]}\}$ is uniformly integrable, it follows that
				$$
				\lim_{M \to \infty} \mathbb{E}^{\mathbb{P}} \left( \int_0^T \mathds{1}_{\{|X_t| > M\}} (1 + |X_t|^2) \, \mathrm{d}t \right) = 0.
				$$
				Hence, $\lim\limits_{n\to\infty} I_n=0$. We conclude from  Gronwall's lemma.\\

				\item[ii)] We compute the following difference
				\begin{align*}
					J(\gamma^n) - J(\gamma) 
					&=\ \mathbb{E}^{\mathbb{P}} \left( \int_{0}^{T} \left( f(s, X^n_s, \mathcal{L}_{X^n_s}, u^n_s) - \int_{\mathbb{A}} f(s, X_s, \mathcal{L}_{X_s}, a)\, q_s(\mathrm{d}a) \right) \mathrm{d}s \right)\\
					&+\mathbb{E}^{\mathbb{P}} \left( \int_{0}^{T} \left( h(s, X^n_s, \mathcal{L}_{X^n_s})\, \mathrm{d}|k^n|_s - h(s, X_s, \mathcal{L}_{X_s})\right)\, \mathrm{d}|k|_s \right) \\
					&+\mathbb{E}^{\mathbb{P}} \left( g(X^n_T) - g(X_T)  \right).
				\end{align*}
				
				Using the result from point \textit{i)} and the same type of arguments as in the proof of Theorem \ref{theooptimal}, we conclude that
				$$
				\lim_{n \to \infty} J(\gamma^n) = J(\gamma).$$
			\end{itemize}
			
		\end{proof}
		\section{Appendix}
		We recall the definition of the Skorokhod problem with càdlàg inputs, reflected in a time-dependent, convex, and bounded domain.
		\begin{definition}\label{sk_dett}
			Let  $D=\{ D_t, t\in[0,T]\}$ be a family of time-dependent sets such that for every $t$,  $D_t\in\mathcal{C}$ and $t\to D_t$ is càdlàg with respect to Hausdorff metric. Let  $y\in\mathbb{D}([0,T],\mathbb{R}^d)$ such that
			$ y_{0}\in \mathrm{D}_0$. We say that a pair of functions $(x,k)\in\mathbb{D}([0,T],\mathbb{R}^{2d})$  is a solution to the reflection problem associated with $y$ if:
			
			\begin{itemize}	
				\item[i)]  $x_t=y_t+k_t,\; \text{for every} \;t \in [0,T]$,
				\item[ii)] $x_t \in D_t, \; \text{for every} \; t \in [0,T]$,
				\item[iii)] For every càdlàg $z$, such that $z_t\in D_t$, we have  for every $t\in [0,T]$
				$$\int_0^t \langle x_{s} - z_s,  \mathrm{d}k_s \rangle \leq 0.$$
			\end{itemize}
			This problem will be denoted by $RP_{D}(y)$.
		\end{definition}
		\begin{theorem}\label{prop_A}\cite{jarni25}
			Let  $D=\{ D_t, t\in[0,T]\}$ be a family of time-dependent sets such that for every $t\in [0,T]$,  $D_t\in\mathcal{C}$,  and $t\mapsto D_t$ is càdlàg with respect to Hausdorff metric. Let $y$ and $a$ be elements of  $\mathbb{D}([0,T],\mathbb{R}^d)$ such that $ y_{0}\in D_0$, for every $t\in[0,T]$, $a_t\in \mathring{D_t}$, and $\inf_{t\in[0,T]}d(a_t,\partial D_t)>0$.
			Then, there exists a unique solution of the problem $RP_{D}(y)$. In addition,
			\begin{equation}
				|k|_{T} \leq M( \sup_{s\in [0,T]}|y_s|^2+ \sup_{s\in [0,T]}|a_s|^2 +1). \label{estim_k_}
			\end{equation}
		\end{theorem}
		\begin{proposition}\label{prop_B}\cite{jarni25}
			Let $\mathcal{D}=\{ D_t, t\in[0,T]\}$ and $\mathcal{\tilde{D}}=\{ \tilde{D}_t, t\in[0,T]\}$ be  two families of time-dependent sets such that for every $t\in [0,T]$, $\mathcal{D}_t$ and $\mathcal{\tilde{D}}_t$ belong to $\mathcal{C}$, $t\mapsto D_t$ and $t\mapsto \tilde{D}_t$ are càdlàg with respect to Hausdorff metric, $D_T=D_{T^-}$, and $\tilde{D}_T=\tilde{D}_{T^-}$. Let $y$ and $\tilde{y}$ be elements of $\mathbb{D}([0,T],\mathbb{R}^d)$ such that $y_0\in D_0$ and $\tilde{y}_0\in\tilde{D}_0$. Let $(x,k)$ and $(\tilde{x},\tilde{k})$ be the solutions associated to $RP_\mathcal{D}(y)$ and $RP_\mathcal{\tilde{D}}(\tilde{y})$, respectively.
			Then, for every $t\in [0,T]$, we have
			\begin{equation}
				|k_t-\tilde{k}_{t}|^2\leq 2\left(|y_t-\tilde{y}_{t}|^2 + \sup_{s\in [0,t]} d_H(D_s,\tilde{D}_s)(|k|_t+|\tilde{k}|_t)+\int_{0}^{t}\langle y_t-y_s+\tilde{y}_s-\tilde{y}_t,\mathrm{d}(k_s-\tilde{k}_s) \rangle\right).\label{ineq_estim}
			\end{equation}
		\end{proposition}
		
		\begin{proposition}\label{estim_quad_Ap}
			Let $\mathcal{D}=\{ D_t,\; t\in[0,T]\}$ be a family of time-dependent sets such that for every $t\in [0,T]$, $D_t\in\mathcal{C}$ and $t\mapsto D_t$ is càdlàg with respect to the Hausdorff metric. Let $Y$ and $\tilde{Y}$ be two semimartingales with the following decompositions: $Y=M+V$ and $\tilde{Y}=\tilde{M}+\tilde{V}$, such that $Y_{0},\tilde{Y}_{0}\in D_0$. Let $p\geq 1$ and assume that $M$ and $\tilde{M}$ are martingales in $\mathcal{S}_{0,T}^{2p}$, $V$ and $\tilde{V}$ are càdlàg adapted processes with finite variation such that $|V|$ and $|\tilde{V}|$ have a finite $2p$-th moment.
			Let $(X,K)$ and $(\tilde{X},\tilde{K})$ be the solutions associated with $RP_{\mathcal{D}}(Y)$ and $RP_{\mathcal{D}}(\tilde{Y})$, respectively. Then, there exists a constant $C_p$ such that for all stopping times $\tau$ in $[0,T]$, we have
			\begin{multline*}
				\mathbb{E}\left(\sup_{0\leq t\leq \tau}|X_{t}-\tilde{X}_{t}|^{2p}\right)+\mathbb{E}\left(\sup_{0\leq t\leq \tau}|K_{t}-\tilde{K}_{t}|^{2p}\right)\\\leq C_p\mathbb{E}\left(|X_0-\tilde{X}_{0}|^{2p}+[M-\tilde{M}]_{\tau}^p+|V-\tilde{V}|_{\tau}^{2p}\right).
			\end{multline*}
		\end{proposition}
		\begin{proof}
			When $p=1$ the result is shown in Proposition \eqref{estim_quad} in \cite{jarni25}. It can be extended for $p\geq 1$ by similar arguments as in the proof of Proposition \eqref{estim_quad}.
		\end{proof}

	\end{document}